\documentclass{amsart}
\usepackage{amsmath}%
\usepackage{amsfonts}%
\usepackage{amssymb}%
\usepackage{graphicx}
\newtheorem{theorem}{Theorem}
\theoremstyle{plain}

\newtheorem{conjecture}{Conjecture}
\newtheorem{corollary}{Corollary}

\newtheorem{definition}{Definition}
\newtheorem{example}{Example}
\newtheorem{lemma}{Lemma}

\newtheorem{remark}{Remark}

\newtheorem{assumption}{Assumption}
\numberwithin{equation}{section}

\renewcommand{\rho}{\varrho}

\newcommand\M{\, {\mathcal M}\, }

\newcommand\R{{\mathbb R}}

\newcommand\bp{{\bar p}}
\newcommand\bq{{\bar q}}
\newcommand{\ii}{\underline 1}

\DeclareMathOperator*{\argmax}{arg\,max}
\DeclareMathOperator*{\argmin}{arg\,min}
\newcommand\BR{{\mathcal B \! \mathcal R \!}}
\begin{document}
\title[Piecewise affine Hamiltonian flows]{Piecewise linear 
Hamiltonian flows associated to zero-sum games:
transition combinatorics and questions on ergodicity}
\author{Georg Ostrovski and Sebastian van Strien}
\address{Mathematics Department, University of Warwick}
\email{G.Ostrovski@warwick.ac.uk \,\,\, strien@maths.warwick.ac.uk}
\begin{abstract}
In this paper we consider a class of piecewise affine Hamiltonian 
vector fields whose orbits are piecewise straight lines. We give a first classification result of such systems
and show that the orbit-structure of the flow of such a differential
equation is surprisingly rich.
\end{abstract}
\thanks{The authors would like to thank Abed Bounemoura and Vassili Gelfreich for several valuable comments.}
\maketitle


\section{Introduction}
Traditionally, the motivation for studying Hamiltonian systems comes either from 
classical mechanics or from partial differential equations. In this paper we 
study a particular class of Hamiltonian systems which are motivated in a completely
different way: they arise from  dynamical systems associated to game theory. 
This class of Hamiltonian systems was proposed in \cite{VanStrien2009a}
and corresponds to Hamiltonians which are piecewise affine. In this paper
we study several properties of such dynamical systems.

For general smooth functions $H\colon \R^{2n}\to \R$ for which
$H^{-1}(1)$ is (topologically) a $2n-1$ sphere, very little is known about the global dynamics of the
corresponding Hamiltonian vector field, even if we make convexity assumptions
on the level sets. If $H$ is a quadratic function,
the system is completely integrable and the flow is extremely simple.
So instead, in this paper we consider Hamiltonians $H$
which are piecewise affine and investigate whether the following 
analogy holds:
\\ \\
\begin{tabular}{ |  l | l |  }
\hline   circle diffeomorphism & circle rotation  \\
quadratic map $x\mapsto 1-ax^2$ & tent map $x\mapsto 1-a|x|$ \\
H\'enon map $(x,y)\mapsto (1-ax^2+by,x)$ & Lozi map   $(x,y)\mapsto (1-
a|x|+by,x)$ \\
Smooth Hamiltonian  & Piecewise affine Hamiltonian  \\ 
Smooth area preserving maps & Piecewise affine area preserving maps\\
\hline 
\end{tabular}
\\

\noindent This analogy would suggest that one might gain insight about smooth Hamiltonian systems by looking at
piecewise affine ones, in the same way as circle diffeomorphisms can be modelled by circle rotations. Much of the complexity of the dynamics may well persist in the piecewise affine case, and even though these systems are not smooth
they still may be easier to study. 

This paper consists of two parts.  In the first part  we prove a classification theorem for the case
when $n=2$ which goes towards a description of the global dynamics in terms of coding, 
while in the second part we give numerical results which show 
that much richness of the dynamics in the smooth case  persists. 

\medskip

Let us be more precise. We say that $H\colon \R^{2n} \to \R$ is a {\em piecewise affine} function if 

(i) $H$ is continuous and 

(ii) there exists a finite number of hyper-planes $Z_1,\dots,Z_k$ in $\R^{2n}$
so that $H$ is affine on each component of $\R^{2n}\setminus \cup_{i=1}^k Z_i$. 

Since $H$ is piecewise affine, $\left(\dfrac{\partial H}{\partial q},  \dfrac{\partial H}{\partial p}\right)$ are piecewise constant outside the hyper-planes  $Z_i$ and {\em there} the derivatives are multivalued. Consider the Hamiltonian system
\begin{equation}
\dfrac{dp}{dt}\in \dfrac{\partial H}{\partial q}, \dfrac{dq}{dt}\in -
\dfrac{\partial H}{\partial p},\label{eq-1}
\end{equation}
or more generally the Hamiltonian vector field $X_H$ associated to $H$ and the symplectic 2-form $\sum_{ij}\omega_{ij} dp_i\wedge dq_j$ (where $(\omega_{ij})$ are the (real) coefficients of some constant 
non-singular $n\times n$ matrix $\Omega$) and the corresponding differential inclusion
\begin{equation}
(\dfrac{dp}{dt},\dfrac{dq}{dt})\in X_H(p,q).\label{eq0}
\end{equation}
(Here $X_H$ is defined by requiring that $\omega(X_H,Y)=dH(Y)$ for each vector field
$Y$.)
In other words,  
\begin{equation}
\left(\begin{array}{c}\dfrac{dp}{dt} \\  \\ \dfrac{dq}{dt}\end{array}\right)\in  \left( \begin{array}{cc}  0 & \Omega \\ -\Omega' & 0 \end{array}\right) \left(\begin{array}{c}\dfrac{\partial H}{\partial p} \\  \\ \dfrac{\partial H}{\partial q}\end{array}\right),\label{eq0'}
\end{equation}
where $\Omega'$ stands for the transpose of the matrix $\Omega$. The reason we write $\in$ rather than $=$ in the above equations is because $\dfrac{\partial H}{\partial q}, \dfrac{\partial H}{\partial p}$ are multivalued. In this generality, there is no reason to assume that the flow of these differential inclusions is continuous.

Instead we will consider the special case of $H\colon \Sigma\times \Sigma\to \R$ defined by
\begin{equation}
H(p,q)= \max_i \, (\M  q)_i -  \min_j \, (p'  \M)_j\,\, ,
\label{defH}
\end{equation}
where $p'$ stands for the transpose of $p$, $(\M q)_i$ and $(p' \M)_j$ stand for the $i$-th and $j$-th component of the vectors $\M q$ respectively $p' \M$ and $\Sigma$ consists of the set of probability vectors in $\R^n$. Here we consider the corresponding Hamiltonian system\footnote{Differential equations with discontinuities along a hyperplane are often called 'Filippov systems', and there is a large literature on such systems, see for example~\cite{MR2368310}, \cite{MR1789550} and \cite{MR2103797}. The special feature of the systems we consider here is that they have discontinuities along $n\cdot(n-1)$ intersecting hyperplanes in $\Sigma \times \Sigma$.} defined on the space $\Sigma \times \Sigma$. To ensure that level 
sets are simple,  we make the following {\bf assumption} on $\M$: there exist $\bp,\bq\in \Sigma$ so that all its components are strictly positive and so that $\bp' \M=\lambda \ii'$ and $\M \bq=\mu \ii$ for some $\lambda,\mu\in \R$ where $\ii=(1\, 1 \, \dots \, 1)\in \R^n$.
A motivation for this assumption and an interpretation of $\bp,\bq$ is given in the paragraph above equation (\ref{eq:pi_proj}).
As was shown in \cite{VanStrien2009a}, one then has the following properties:
\begin{enumerate}
\item For an open set of full Lebesgue measure of $n\times n$
matrices $\M$, {\em each} level set $H^{-1}(\rho)$, $\rho>0$ is topologically a $(2n-3)$-dimensional sphere
and $H^{-1}(0)=\{(\bp,\bq)\}$ (note that $\dim(\Sigma\times \Sigma)=2n-2$). In fact, $H^{-1}(\rho)$ bounds a convex ball. 
\item For an open set of full Lebesgue measure of $n\times n$ matrices $\Omega$ and $\M$, there exists a unique solution of (\ref{eq0}) for all initial conditions and the corresponding flow $(p,q,t)\mapsto \phi_t(p,q)$ is continuous, which is not generally true for differential inclusions. 
\item The flow $(p,q,t)\mapsto \phi_t(p,q)$ is piecewise a translation flow, and first return maps to hyperplanes in $H^{-1}(\rho)$ are piecewise affine maps.
\end{enumerate}
In the second part of this paper we will present some numerical studies
of examples of such systems when $n=3$, i.e., 
when each energy level $H^{-1}(\rho)$, $\rho>0$, is  topologically a three-sphere.
Let us single out one of these examples: there exists an
open set of $3\times 3$ matrices $\M$ and $\Omega$ with the following property.
For each $\rho>0$ there exists a topological disc $D\subset H^{-1}(\rho)$ 
which is made up of 4 triangles in $\Sigma\times \Sigma$
so that \begin{enumerate}
\item the first return map $R$ to $D$ is continuous and extends continuously to the closure of $D$;
\item $R|\partial D=id$ (in fact $\partial D$ corresponds to a periodic orbit of the flow, whose Floquet multipliers are undefined);
\item $R\colon D\to D$ is area-preserving and piecewise affine;
\item each orbit in $H^{-1}(\rho)$ intersects $D$ infinitely often.
\end{enumerate}
Moreover, $R\colon D\to D$ contains hyperbolic horseshoes and also an elliptic periodic orbit, surrounded by an elliptical disk consisting of quasi-periodic orbits on invariant ellipses, see \cite{VanStrien2009a,VanStrien2009b}. Numerical simulations for these systems suggest that all orbits outside this elliptical disk in $D$ are dense in this set, see  Example 2 in Part 2 of this paper. This first return map could be a piecewise affine model for general smooth area preserving maps of the disk. 

As mentioned, one motivation for   looking at Hamiltonians as in (\ref{defH}) and the corresponding systems (\ref{eq0})
comes from game theory. Indeed, these correspond to certain differential inclusions which are naturally associated to so-called \textit{zero-sum games} in game theory. For this, let $A,B$ be $n \times n$ matrices and define $\BR_A(q):=\argmax_{p\in \Sigma} p'A q$ and $\BR_B(p):=\argmin_{q\in \Sigma} \, p' B q$ and consider the differential inclusion:
\begin{equation}
\dfrac{dp}{dt}\in\BR_A(q)-p, \dfrac{dq}{dt}\in \BR_B(p)-q.
\label{eqBR}
\end{equation}
The zeros of this equation, i.e.  the set of points $(\bp,\bq)$ for which $\bp \in \BR_A(\bq), \bq \in \BR_B(\bp)$, are called {\em Nash equilibria}. The differential inclusion (\ref{eqBR}) is called the {\em Best Response Dynamics} and the corresponding time-rescaled 
\begin{equation*}
\dfrac{dp}{dt}\in\dfrac{1}{t}(\BR_A(q)-p), 
\dfrac{dq}{dt}\in \dfrac{1}{t}(\BR_B(p)-q),
\end{equation*}
the {\em Fictitious Play Dynamics} associated to the game with matrices $A$ and $B$. In game theory and economics these are often used to {\em model learning}, i.e. to describe how people {\lq}learn{\rq} to play a game. It turns out that in the case of zero-sum games, i.e. when  $A=-B$, all orbits of (\ref{eqBR}) converge to the set of Nash equilibria, see \cite{Robinson1951} and for a short proof see \cite{Hofbauer1995}.

Note that  if we define $\M:=A=-B$
then the existence of $\bp,\bq\in \Sigma$ for which all coordinates
are strictly positive and for which $\bp' \M=\lambda \ii'$ and $\M \bq=\mu \ii$ 
(as assumed just below
(\ref{defH})) implies that $0\in \BR_A(\bq) - \bp$ and $0\in \BR_B(\bp) - \bq$.
So such a point $(\bp,\bq)\in \Sigma\times \Sigma$ is a Nash equilibrium.
To see this, notice that $\BR_{\M}(\bq)$ is the convex hull of all the unit vectors
corresponding to the largest component(s) of $\M\bq$. So if all components of $\bp,\bq$ 
are strictly positive, then $0\in \BR_A(\bq)-\bp$  holds iff all coordinates of $\M\bq$
are equal. 

We should note that $p$ and $q$ have a {\em different connotation} here from the usual one
in classical mechanics: $p$ corresponds to the position (the probablity  vector describing past play) of the first player
and $q$ the corresponding object for the second player.  Although the equation (\ref{eqBR}) itself is
not Hamiltonian at all, 
it is closely related to the Hamiltonian system (\ref{eq0}). Indeed, for  $(p,q)\in \Sigma\times \Sigma\setminus \{\bp,\bq\}$, let $l(p,q)$ be the half-line from $(\bp,\bq)$ containing $(p,q)$ and define 
\begin{equation}\label{eq:pi_proj}
\pi\colon \Sigma\times \Sigma\setminus \{\bp,\bq\} \rightarrow H^{-1}(1),
\end{equation}
where $\pi(p,q)\in H^{-1}(1)$ is the intersection of $l(p,q)$ with the $2n-3$-dimensional sphere $H^{-1}(1)$. It turns out that the projection of the flow on $H^{-1}(1)$ corresponds to the solution of a Hamiltonian system as above. In other words, the Hamiltonian dynamics describes the {\lq}spherical coordinates{\rq}. So we will think of the dynamics of (\ref{eqBR}) as {\em inducing} Hamiltonian dynamics. For more details see Section \ref{sec:ham_br_dynamics}.

In this paper we will study Hamiltonian dynamics coming from Best Response dynamics as in (\ref{eqBR}) in the case where $n=3$.
The aim of this paper is the following:
\begin{itemize}
\item Because of the special nature of the Hamiltonian systems we consider in this paper, one can associate in a natural way itineraries to each orbit. In the Main Theorem of this paper we will show that not all itineraries are possible and we will give a full classification of all possible transition diagrams\footnote{A transition diagram shows all allowable transitions within orbits (but not all allowable transitions correspond to actual orbits).}.
\item The Hamiltonian dynamics appearing in this paper is much simpler than usual: the first return maps to planes are piecewise translations. In the second part of this paper we will show numerical simulations concerning a number of examples of such systems. The four examples which we show here, display (conjecturally):
\begin{enumerate}
\item fully ergodic behaviour;
\item elliptic behaviour of a very simple type (which
we can prove rigorously, see \cite{VanStrien2009a});
\item elliptic behaviour of a {\lq}composite{\rq} type;
\item Arnol'd diffusion, and intertwining of various elliptic regions.
\end{enumerate}
\item We believe that the elliptic behaviour occuring in our systems satisfies a huge amount of regularity, and we will formulate several questions and conjectures formalizing this.
\end{itemize}

\part{Combinatorial results}

In the first part of the paper we introduce the Best Response and Fictitious Play Dynamics inducing a special case of the Hamiltonian Dynamics presented above. We introduce a combinatorial description of the BR dynamics and provide a combinatorial characterisation of the dynamics for zero-sum games for $n = 3$ (inducing Hamiltonian dynamics with two degrees of freedom).

\section{Hamiltonian and Best Response Dynamics}
\label{sec:ham_br_dynamics}

We define $\Sigma_A$ to be the simplex of probability row vectors in $\mathbb{R}^n$ and $\Sigma_B$ the simplex of probability column vectors in $\mathbb{R}^n$: 
\begin{align*}
\Sigma_A = \left\{ p \in \mathbb{R}^{1 \times n} : p_i \geq 0, \sum p_i = 1 \right\},
\Sigma_B = \left\{ q \in \mathbb{R}^{n \times 1} : q_i \geq 0, \sum q_i = 1 \right\}.
\end{align*}
We denote $\Sigma = \Sigma_A \times \Sigma_B$.

We consider a bimatrix $(A,B)$, where $A=(a_{ij}),B=(b_{ij}) \in \mathbb{R}^{n \times n}$. With a slight abuse of notation we identify the standard unit vector $e_k$ with the integer $k$ and define the following correspondences\footnote{In a game theoretic context the matrices are \textit{payoff matrices} of two players A and B, where each player has $n$ pure strategies, i.e. rows or columns, to choose from. $\Sigma_A$ and $\Sigma_B$ are the spaces of \textit{mixed strategies}, i.e. probability distributions over the $n$ \textit{pure strategies} that are represented as the standard unit vectors. $\BR_A$ and $\BR_B$ are the two players' respective \textit{best-response correspondences}, assigning a payoff-maximising pure strategy answer to the strategy played by the opponent.}:

\begin{align*}
&\BR_A(q) = \argmax_{p\in \Sigma_A} p A q = \argmax_i \left( A q\right)_i \text{ for } q \in \Sigma_B,\\
&\BR_B(p) = \argmax_{q\in \Sigma_B} p B q = \argmax_j \left( p B\right)_j \text{ for } p \in \Sigma_A.
\end{align*}

These correspondences are singlevalued almost everywhere, except on a finite number of hyperplanes. On these hyperplanes (also called \textit{indifference planes}) at least one of the $\BR$ correspondences has as its values the set of convex combinations of two (ore more) unit vectors. 

We can now define the \textit{Fictitious Play Dynamics} as a continuous time dynamical system in $\Sigma$:
\begin{align}\tag{FP}\label{eq:fp_dynamics}
\begin{split}
&\frac{dp}{dt} \in \frac{1}{t}\left( \BR_A(q(t)) - p(t) \right), \\
&\frac{dq}{dt} \in \frac{1}{t}\left( \BR_B(p(t)) - q(t) \right),
\end{split}
\end{align}
for $t > 1$ and some given initial value $(p(1), q(1)) \in \Sigma$. Note that the right hand side is singlevalued almost everywhere.

Although (\ref{eq:fp_dynamics}) is more common in game theory, where it serves as a model of \textit{myopic learning}, we prefer to consider the following time reparametrisation, referred to as the \textit{Best Response Dynamics}:
\begin{align}\tag{BR}\label{eq:br_dynamics}
\begin{split}
&\frac{dp}{ds} \in \BR_A(q(s)) - p(s) , \\
&\frac{dq}{ds} \in \BR_B(p(s)) - q(s) .
\end{split}
\end{align}
The orbits of both systems coincide (differing only in time parametrisation) but (\ref{eq:br_dynamics}) has the advantage of being autonomous. The existence of solutions for any initial conditions follows from upper semicontinuity of $\BR_A$ and $\BR_B$, see \cite[Chapter 2.1]{Aubin1984}.

The classical learning processes (\ref{eq:fp_dynamics}) and (\ref{eq:br_dynamics}) are naturally related to Hamiltonian Dynamics via the following construction. Let $H \colon \Sigma_A \times \Sigma_B \rightarrow \R$ be defined as
\begin{equation}
H(p,q) = \max_i \left(Aq\right)_i - \min_j\left(pA\right)_j.
\end{equation}
Further let $\Omega = ( \omega_{ij} ) $ be some non-singular $n \times n$ matrix. Consider a Hamiltonian vector field $X_H$ associated to $H$ and the symplectic 2-form $\sum_{ij}\omega_{ij} dp_i\wedge dq_j$. The corresponding differential inclusion is 
\begin{equation}
(\dfrac{dp}{dt},\dfrac{dq}{dt})\in X_H(p,q).
\label{eq3}
\end{equation}
Let us denote $T \Sigma_A = T \Sigma_B = \left\{ v \in \mathbb{R}^n : \sum v_i = 0 \right\}$, and let $\Omega', A'$ be the transpose of the matrices $\Omega$ and $A$. Further let $P_A\colon \R^n \rightarrow T \Sigma_A , P_B\colon \R^n \rightarrow T \Sigma_B$ be the parallel projections to $T \Sigma_A$ and $T \Sigma_B$ along the vectors $\Omega'^{-1} \ii$ and $\Omega^{-1} \ii$ respectively (where $\ii = (1 1 \ldots 1) \in \R^n$). Then a simple calculation (see \cite{VanStrien2009a}) shows that the above inclusion (\ref{eq3}) takes the form
\begin{align}\label{eq:ham_eq}
\begin{split}
\dfrac{dp}{dt} &\in P_A \Omega'^{-1} A' \BR_A(q),\\
\dfrac{dq}{dt} &\in P_B \Omega^{-1} A \BR_B(p).
\end{split}
\end{align}
The projections $P_A,P_B$ appear in these equations because $H$ is considered
as a function on $\Sigma = \Sigma_A \times \Sigma_B$ and so the dynamics is constrained to 
this affine subspace. The Hamiltonian differential inclusion (\ref{eq:ham_eq})
is closely related to the BR dynamics if we take  $\Omega=A$.
In this case the dynamics defined on $H^{-1}(1)$ by (\ref{eq:ham_eq}) equals the BR-dynamics projected via $\pi$ (defined in equation (\ref{eq:pi_proj})) to this level set of $H$ (for details see \cite{VanStrien2009a}).
In other words, if we compute an orbit under the BR-dynamics then the image under $\pi$ of this orbit
is an orbit under the Hamiltonian dynamics (\ref{eq:ham_eq}). 

For this reason, for the rest of this paper we will assume $\Omega = A$, where $A \in \mathbb{R}^{3\times 3}$ is non-singular. 

\medskip

For later use, let us make some simple observations. 
$\Sigma_B$ can be divided into $n$ convex regions $R^A_i = \BR_A^{-1}(e_i)$, where $i \in \left\{ 1, \ldots, n\right\}$ and analogously $R^B_j = \BR_B^{-1}(e_j) \subset \Sigma_A$, $j \in \left\{ 1, \ldots, n\right\}$. 

\begin{figure}
\begin{center}
\includegraphics{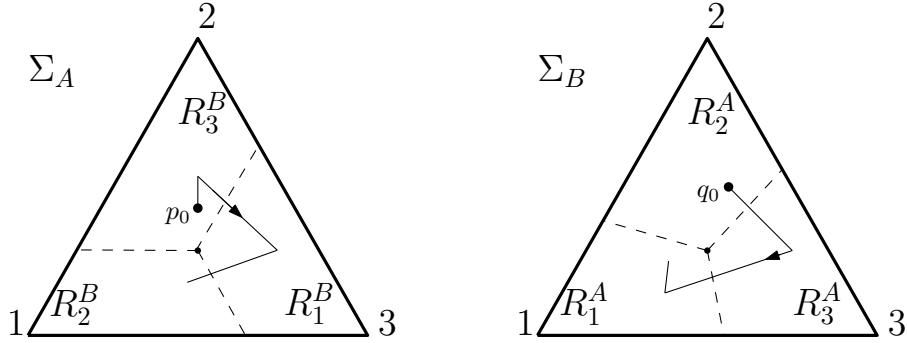}
\end{center}
\caption{An example of $\Sigma_A, \Sigma_B \subset \mathbb{R}^3$  being partitioned into regions $R^B_j$ and $R^A_i$. The dashed lines indicate the hyperplanes at which one of the $\BR$ correspondences is multivalued. Also shown is a piece of an orbit with initial conditions $(p_0,q_0)$. The orbit changes direction four times: when $p$ crosses a dashed line, $q$ changes direction, and vice versa. Note that this is an orbit for the BR dynamics on $\Sigma = \Sigma_A \times \Sigma_B$, not for the induced Hamiltonian dynamics on $H^{-1}(1)$.}
\label{fig:fp_example}
\end{figure}

Since $\BR_A \times \BR_B$ is constant on $R_{ij} = R^B_j \times R^A_i$, (\ref{eq:fp_dynamics}) and (\ref{eq:br_dynamics}) have continuous orbits which are piecewise straight lines heading for vertices $(e_k,e_l) \in \Sigma$ whenever $(p(t),q(t)) \in R_{kl}$. The orbits only change direction at a finite number of hyperplanes, namely whenever $\BR_A$ or $\BR_B$ (or both) become multivalued. More precisely, $p(t)$ changes direction whenever $q(t)$ passes from $R^A_i$ to $R^A_{i'}$ for some $i \neq i'$, and vice versa. See Fig.\ref{fig:fp_example} for an example with $n = 3$.

\section{Combinatorial Description for the Case $n = 3$}

In this section and later on we restrict our attention to the case of dimension $n = 3$. The partition of $\Sigma$ into the convex blocks $R_{ij}$ quite naturally gives rise to coding of orbits of (\ref{eq:br_dynamics}) and (\ref{eq:fp_dynamics}). We codify an orbit $(p(t),q(t))$ by a (finite or infinite, one-sided) itinerary $(i_0,j_0)\rightarrow (i_1,j_1)\rightarrow \ldots \rightarrow (i_k,j_k)\rightarrow \ldots$ indicating that there exists a sequence of times $(t_k)$, such that $(p(t),q(t)) \in R_{i_k,j_k}$ for $t_k < t < t_{k+1}$. To simplify notation we will often write $(i,j)$ instead of $R_{ij}$. 

\begin{figure}
\begin{center}
\includegraphics[width = 0.5\textwidth]{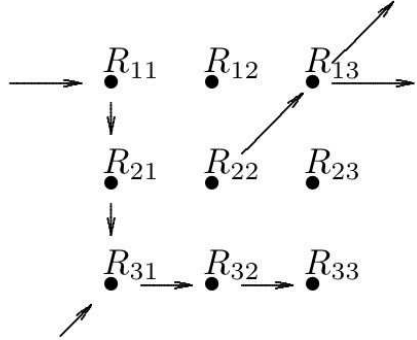}
\end{center}
\caption{A hypothetical graph representing transitions for the itineraries of orbits of BR dynamics. We will see later that this very graph cannot be a transition graph for a bimatrix game.}
\label{fig:markov}
\end{figure}

Abstractly we have a graph of nine vertices $(i,j),~i,j=1,2,3$ with directed edges between them (see Fig.\ref{fig:markov}).

It is not difficult to see that for an orbit of BR dynamics for a fixed bimatrix $(A,B)$, such itinerary can then contain $(i,j) \rightarrow (i',j')$, $i \neq i'$ (or $j \neq j'$) if and only if $a_{i'j} \geq a_{ij}$ ($b_{ij'} \geq b_{ij}$). Note that for almost all initial conditions in the corresponding orbits $p$ and $q$ never switch directions simultaneously, i.e. the itinerary only contains transitions of the form $(i,j) \rightarrow (i',j)$ and $(i,j) \rightarrow (i,j')$, $i \neq i'$, $j \neq j'$ (see \cite{Sparrow2007}). 

Further we want to make sure that for any $i \ne i'$ ($j \ne j'$) there is only one possible transition direction between $(i,j)$ and $(i',j)$ ($(i,j)$ and $(i,j')$). Therefore we introduce the following non-degeneracy assumption on the bimatrix $(A,B)$:

\begin{assumption} \label{as:non_deg}
$a_{ij} \neq a_{i'j}$ and $b_{ij} \neq b_{ij'}$ for all $i,i',j,j'$ with $i \neq i'$ and $j \neq j'$. 
\end{assumption}
Clearly the set of bimatrices satisfying this assumption is open dense with full Lebesgue measure in the space of bimatrices. 

The possible transitions can be expressed in a \textit{transition diagram} as in Fig.\ref{fig:diag_intro}(a). The three rows and three columns of the diagram represent the regions $R^A_i,~i=1,2,3$ and $R^B_j,~j=1,2,3$, respectively. The arrows indicate the possible transitions between the regions, which by Assumption \ref{as:non_deg} always have a unique direction. For example, the itinerary of an orbit of the BR dynamics for a given bimatrix can contain $(1,2) \rightarrow (1,3)$ if and only if in the first row of its transition diagram an arrow points from the second into the third column. Opposite sides of the diagram should be thought of as identified, so that possible transitions between the first and third rows and columns are indicated by arrows on the boundary of the diagram. It is important to note that this partition does not have the nice properties of a Markov partition: there is no claim that every itinerary that can be obtained from the transition diagram can actually be realised by an orbit of the BR dynamics.

\begin{figure}
\begin{center}
\includegraphics[scale=0.8]{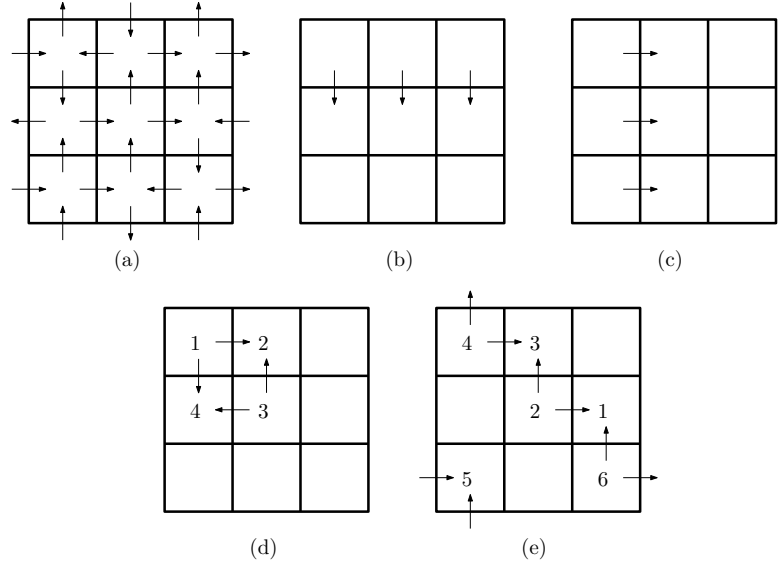}
\end{center}
\caption{(a) example of a transition diagram 
\newline (b) dominated row  
\newline (c) dominated column  
\newline (d),(e) examples of alternating cycles (which will be shown to be impossible in zero-sum games)}
\label{fig:diag_intro}
\end{figure}

One may now ask whether a given transition diagram is realisable as the transition diagram of BR dynamics for a bimatrix game $(A,B)$, and how properties of a game relate to the combinatorial information given by its transition diagram. 

A simple first observation is that no row of a transition diagram can have three horizonal arrows pointing in the same direction, as this would imply $a_{i,j} > a_{i,j'} > a_{i,j''} > a_{i,j}$. Analogously no column of such diagram can have three vertical arrows pointing in the same direction. 

It is easy to see that apart from this restriction, any transition diagram can be realised by appropriate choice of $(A,B)$. However, our interest lies in bimatrix games whose BR dynamics displays 'non-trivial' behaviour. For this we introduce the following assumption:
\begin{assumption}\label{as:dom_strat}
In the transition diagram no row (column) is dominated by another row (column), i.e. no three vertical (horizontal) arrows between two rows (columns) point in the same direction (Fig.\ref{fig:diag_intro}(b)-(c))
\footnote{In game theoretic terms this means that none of the players has any 'strictly dominated pure strategy'. The existence of such strategy would mean that the dynamics essentially reduces to the FP/BR dynamics of a $2 \times 3$ game, which under mild genericity conditions is known to converge to the set of Nash Equilibria in a rather simple way, see \cite{Berger2005}.}.
Formally, let $\{i,i',i''\} = \{j, j', j''\} = \{1,2,3\}$, then
\begin{equation*}
(i,j) \rightarrow (i',j) \text{ and } (i,j') \rightarrow (i',j') \Rightarrow (i',j'') \rightarrow (i,j'')
\end{equation*}
and 
\begin{equation*}
(i,j) \rightarrow (i,j') \text{ and } (i',j) \rightarrow (i',j') \Rightarrow (i'',j') \rightarrow (i'',j).
\end{equation*}
\end{assumption}

\section{Zero-Sum Games}

In game theory, an important class of (bimatrix) games are the zero-sum games, i.e. games $(A,B)$, such that $A+B = 0$. However, our analysis remains valid for a larger class, namely games that are linearly equivalent to a zero-sum game. 

\begin{definition}
Two $3\times 3$ bimatrix games $(A,B)$ and $(C,D)$ are \textbf{linearly equivalent}, if there exist $e>0, f_j, j=1,2,3$ and $g>0, h_i, i = 1,2,3$ such that
\begin{equation*}
c_{ij} = e a_{ij} + f_j \text{ and } d_{ij} = g b_{ij} + h_i .
\end{equation*}
\end{definition}
It can be checked that for linearly equivalent bimatrix games $(A,B)$ and $(C,D)$, the respective best-response correspondences coincide: $\BR_A = \BR_C$ and $\BR_B = \BR_D$. From the definitions it immediately follows that linearly equivalent bimatrix games induce the same dynamics (\ref{eq:fp_dynamics}) and (\ref{eq:br_dynamics}). Since our main focus lies on these dynamical processes, in the rest of this text we call a game $(A,B)$ zero-sum, if there exists a linearly equivalent game $(C,D)$, such that $C+D = 0$.

\begin{definition}\label{def:nash_eq}
A point $(\bp,\bq) \in \Sigma$ is called \textbf{Nash Equilibrium} if $\bp \in \BR_A(\bq)$ and $\bq \in \BR_B(\bp)$. 
\end{definition}

By the definition, Nash Equilibria are precisely the fixed points of (\ref{eq:fp_dynamics}) and (\ref{eq:br_dynamics}). It has been proved by John Nash, that every game with finitely many players and strategies has got a Nash Equilibrium (see \cite{NashJohnForbes1951}). A very important classical result is the following:

\begin{theorem}[Brown \cite{Brown1951}, Robinson \cite{Robinson1951}]\label{thm:zs_conv}
For a zero-sum game, every orbit of the (\ref{eq:br_dynamics}) and (\ref{eq:fp_dynamics}) converges to the set of Nash Equilibria.
\end{theorem}

A short proof using an explicitly given Lyapunov function can be found in \cite{Hofbauer1995}. In the same paper, Hofbauer states the converse conjecture, which still remains open:

\begin{conjecture}[Hofbauer \cite{Hofbauer1995}]\label{conj:hofbauer_conj}
A bimatrix game with a unique Nash Equilibrium point in $\mathring \Sigma$ that is stable under the BR (or FP) dynamics must be a zero-sum game.
\end{conjecture}

As the above indicates, zero-sum games are of great interest in the study of BR and FP dynamics. A natural question to ask is now, which combinatorial configurations (transition diagrams) can be realised by zero-sum games. In this paper we will restrict our attention to zero-sum games with a unique Nash Equilibrium in the interior of $\Sigma$:

\begin{assumption}\label{as:int_ne}
The bimatrix game $(A,B)$ has a unique Nash Equilibrium point $(E^A,E^B)$, which lies in the interior of $\Sigma$. Equivalently (see for instance \cite{VanStrien2009a}), there exists precisely one point $E^B \in \mathring{\Sigma}_B$, such that $(A E^B)_i = (A E^B)_j ~\forall i,j$, and precisely one point $E^A \in \mathring{\Sigma}_A$, such that $(E^A B)_i = (E^A B)_j ~\forall i,j$.
\end{assumption}

\begin{lemma}
For every bimatrix game $(A,B)$, Assumption \ref{as:int_ne} implies Assumption \ref{as:dom_strat}.
\end{lemma}
\begin{proof}
If Assumption \ref{as:dom_strat} does not hold, there is a row (column) in the transition diagram for $(A,B)$ which is dominated by another row (column), say $i$ dominated by $i'$. It follows that $R_i^A$ ($R_i^B$) is empty. 

On the other hand, Assumption \ref{as:int_ne} implies that for all $k,l$, $R_{kl} = R_l^B\times R_k^A$ has non-empty intersection with every neighbourhood of the Nash Equilibrium.
\end{proof}

For later use, we make the following definitions:

\begin{definition}\label{def:alt_cycle}
A sequence $(i_0,j_0),  (i_1,j_1), \ldots , (i_n,j_n) = (i_0,j_0)$ is called \textbf{alternating cycle} for $(A,B)$, if after reversing either all downward and upward pointing or all right and left pointing arrows in the transition diagram of $(A,B)$, it forms a directed loop. Examples of alternating cycles are shown in Fig.\ref{fig:diag_intro}(d) and (e). More formally, either
\begin{itemize}
\item $(i_k,j_k) \rightarrow (i_{k+1},j_{k+1})$ whenever $i_k \neq i_{k+1}$ and $(i_{k+1},j_{k+1}) \rightarrow (i_k,j_k)$ whenever $j_k \neq j_{k+1}$, or
\item $(i_{k+1},j_{k+1}) \rightarrow (i_k,j_k)$ whenever $i_k \neq i_{k+1}$ and $(i_k,j_k) \rightarrow (i_{k+1},j_{k+1})$ whenever $j_k \neq j_{k+1}$.
\end{itemize}
\end{definition}

\begin{definition}
We call $(i,j)$ a \textbf{sink}, if it can be entered but not left by trajectories of the BR dynamics (i.e. $(i',j) \rightarrow (i,j)$ and $(i,j') \rightarrow (i,j)$ for $i' \neq i$ and $j' \neq j$). Conversely we call it a \textbf{source}, if it can be left but not entered. 
\end{definition}

We can now formulate several consequences for the transition diagram of a zero-sum game $(A,B)$:

\begin{lemma} \label{lem:zs_impl}
Let $(A,B)$ be zero-sum and satisfy Assumptions \ref{as:non_deg} and \ref{as:int_ne}. Then:
\begin{enumerate}
\item The transition diagram does not have alternating cycles. 
\item The transition diagram does not have sinks.
\item The transition diagram does not have sources. 
\end{enumerate}
\end{lemma}

\begin{remark}
One can see that in (1) without loss of generality we can only consider cycles in which the $i$- and $j$-component change alternatingly, which justifies the notion of alternating cycle. In fact in $3\times 3$ games (1) reduces to saying that there are no alternating cycles of the two kinds depicted in Fig.\ref{fig:diag_intro}(d) and (e). 
\end{remark}

\begin{figure}
\begin{center}
\includegraphics[scale=0.9]{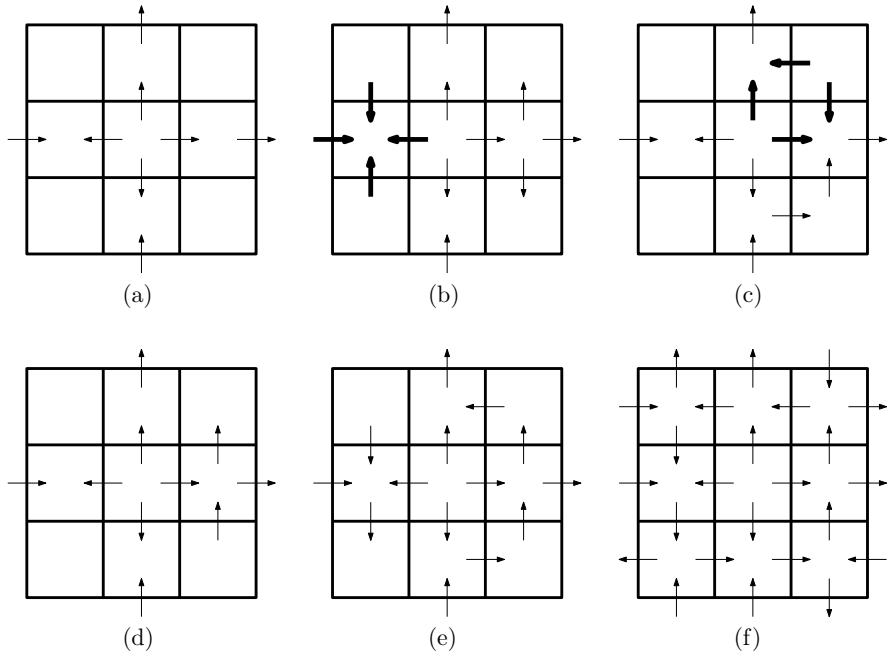}
\end{center}
\caption{Proof of statement (3) of Lemma \ref{lem:zs_impl}: 
\newline (a) General case for a diagram with a source  
\newline (b) Case 1: sink in (2,1) contradicts zero-sum  
\newline (c) Case 2: alternating cycle contradicts zero-sum  
\newline (d) Case 3
\newline (e),(f) Case 3: necessarily following configuration}\label{fig:diag_source}
\end{figure}

\begin{figure}
\begin{center}
\includegraphics{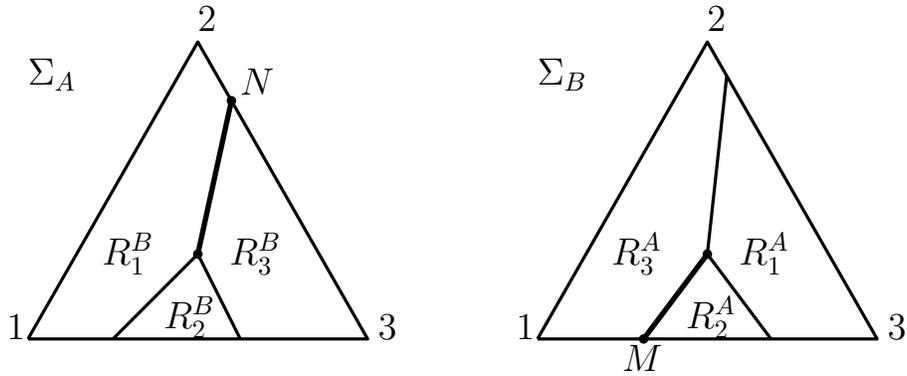}
\end{center}
\caption{Proof of statement (3) of Lemma \ref{lem:zs_impl}:
\newline This configuration necessarily follows from the transition diagram in Case 3 of the proof. $(N,M)$ is then a Nash Equilibrium, contradicting Assumption \ref{as:int_ne}.}
\label{fig:diag_source2}
\end{figure}

\begin{proof}
For statement (1) assume that $A+B = 0$ (otherwise choose linearly equivalent matrices such that this holds and note that this does not change the arrows in the transition diagram). Now note first that $(i,j) \rightarrow (i',j)$ iff $a_{i'j} > a_{ij}$. Further, $(i,j) \rightarrow (i,j')$ iff $b_{ij'} > b_{ij}$, which is equivalent to $a_{ij} > a_{ij'}$. It follows that an alternating cycle leads to a chain of inequalities $a_{i_0j_0} > a_{i_1j_1} > \ldots > a_{i_nj_n} = a_{i_0j_0}$ and therefore cannot exist. 

To prove statement (2), we use Theorem \ref{thm:zs_conv}. It follows from this theorem and Assumption \ref{as:int_ne} that orbits of (\ref{eq:br_dynamics}) converge to a single isolated point in the interior of $\Sigma$. A sink in the transition diagram however would imply that orbits of (\ref{eq:br_dynamics}) that start in $R_{ij}$ (for some $i,j$) stay in it for all times. Since orbits do not change direction while in $R_{ij}$, this can only be the case if they converge (in straight line segments) towards a vertex on $\partial \Sigma$. 

At last, to show statement (3) of the lemma let us assume for a contradiction that such a zero-sum game satisfying the assumptions and with a source in its transition diagram exists. After possibly permuting rows and columns and swapping the roles of the two players we can assume that the source is $(2,2)$ and we have the (incomplete) diagram as seen in Fig.\ref{fig:diag_source}(a). Let us now consider all four possible cases for the vertical arrows in $(2,3)$:

\begin{itemize}
\item Case 1: $(2,3) \rightarrow (i,3),~i=1,2$, i.e. both arrows pointing \textit{out of} $(2,3)$:
\newline Either row 2 dominates row 1 or 3, or $(2,1)$ is a sink, see Fig.\ref{fig:diag_source}(b).

\item Case 2: $(i,3) \rightarrow (2,3),~i=1,2$, i.e. both arrows pointing \textit{into} $(2,3)$:
\newline Either column 2 dominates column 3 or there is an alternating cycle, see Fig.\ref{fig:diag_source}(c). 

\item Case 3: $(3,3) \rightarrow (2,3)$, $(2,3) \rightarrow (1,3)$, both arrows point \textit{upward} (Fig.\ref{fig:diag_source}(d)):
\newline In order to avoid an alternating cycle and a dominated column, one necessarily has $(3,2)\rightarrow(3,3)$ and $(1,3)\rightarrow (1,2)$. Further, since row 2 may not dominate row 1 and alternating cycles cannot happen in a zero-sum game, one gets $(1,1) \rightarrow (2,1)$ and $(1,2) \rightarrow (1,1)$, also $(2,1) \rightarrow (3,1)$ is necessary to avoid a source in $(2,1)$, see Fig.\ref{fig:diag_source}(e). With some further deductions of the same kind one can show that the only possible transition diagram is the one shown in Fig.\ref{fig:diag_source}(f).
We can now deduce that $\Sigma_A$ and $\Sigma_B$ are partitioned into the regions $R_i^A$ and $R_j^B$ as shown in Fig.\ref{fig:diag_source2}. Consider the point on $\partial \Sigma$ denoted by $(N,M)$ and note that $\BR_A(M)$ cotains $e_2$ and $e_3$, hence all their convex combinations. Therefore $N \in \BR_A(M)$. Analogously, $M \in \BR_B(N)$. Therefore $(N,M)$ is a Nash Equilibrium contradicting our assumption that the interior Nash Equilibrium is unique. 

(In fact it also follows from this configuration that there exist initial conditions arbitrarily close to the interior Nash Equilibrium whose trajectories spiral off towards $(N,M)$ and therefore the interior Nash Equilibrium cannot be stable for the dynamics.)

\item Case 4: $(1,3) \rightarrow (2,3)$ and $(2,3) \rightarrow (3,3)$, i.e. both arrows point \textit{downward}:
\newline This case is analogous to the previous one.
\end{itemize}

To conclude, we have shown that a source in the diagram contradicts our assumption of a zero-sum game with unique interior Nash Equilibrium, which finishes the proof of statement (3).

\end{proof}

\section{Main Result}

The next hope is of course to get a full characterization of all combinatorial configurations that can be realised by zero-sum games. This indeed can be done after defining a suitable notion of combinatorially equivalent games. 

\begin{definition}\label{def:comb_equiv}
We call two bimatrix games $(A,B)$ and $(C,D)$ \textbf{combinatorially identical}, if they induce the same transition relation, i.e. $(i,j)\rightarrow (i',j')$ for $(A,B)$ iff $(i,j)\rightarrow (i',j')$ for $(C,D)$.

We call two bimatrix games $(A,B)$ and $(C,D)$ \textbf{combinatorially equivalent}, if there exist permutation matrices $P, Q$ such that $(A,B)$ and $(PCQ, PDQ)$ are combinatorially identical or $(B',A')$ and $(PCQ, PDQ)$ are combinatorially identical.
\end{definition}

The definition expresses the idea, that games are combinatorially equivalent if they have the same transition diagram up to permutation of rows and columns and transposition. The main result is the following:

\begin{theorem}[Combinatorial classification of transition diagrams for zero-sum games] \label{thm:main}
The types of transition diagram (combinatorial equivalence classes) that can be realised by a zero-sum game satisfying Assumptions \ref{as:non_deg} and \ref{as:int_ne} are precisely all those that satisfy the following (combinatorial) conditions:
\begin{enumerate}
\item No row of the diagram has three horizonal arrows pointing in the same direction and no column has three vertical arrows pointing in the same direction. 
\item No three horizontal arrows between two columns point in the same direction and no three vertical arrows between two rows point in the same direction.
\item The diagram has no sinks.
\item The diagram has no sources.
\item The diagram has no alternating cycles. 
\end{enumerate}
This gives precisely 23 different types of transition diagrams (up to combinatorial equivalence). These are listed in Appendix \ref{ap:types}. \footnote{Coincidentally (or not?) the number 23 is the most sacred number for the religion called 'Discordianism'. In this religion 23 is the number of the highest deity, Eris, who is the Greek goddess of Chaos.}
\end{theorem}

Throughout the proof of the theorem we will make use of the following notion:
\begin{definition}\label{def:short_loop}
We call an oriented loop of length four (formed by the arrows in the diagram) a \textbf{short loop}, see Fig.\ref{fig:main_proof}(a). A short loop always has the form $(i,j) \rightarrow (i',j) \rightarrow (i',j') \rightarrow (i,j') \rightarrow (i,j)$ and we indicate the vertex in the diagram encircled by such loop by a $\bullet$.\footnote{It can be checked that a short loop precisely corresponds to those $2\times 2$ subgames, which are linearly equivalent to a zero-sum game.}
\end{definition}

\begin{proof}
By the above discussion we know that (1) is true for any transition diagram of a game and (2) corresponds to Assumption \ref{as:dom_strat} (which is implied by Assumption \ref{as:int_ne}), so the only conditions that are left to check are (3)-(5). By Lemma \ref{lem:zs_impl}, we already know that (3)-(5) are necessary conditions for a diagram to be realisable by a zero-sum game.

To show that (1)-(5) are also sufficient, we will proceed in two steps: we will show that combinatorially these conditions give rise to precisely 23 types of diagrams (up to permutation of rows and columns and transposition) and then we will give examples of zero-sum games realising these types. Because of the initially large number of possible transition diagrams, we will group them by the number of short loops contained in them. 

\begin{figure}
\begin{center}
\includegraphics[scale=0.9]{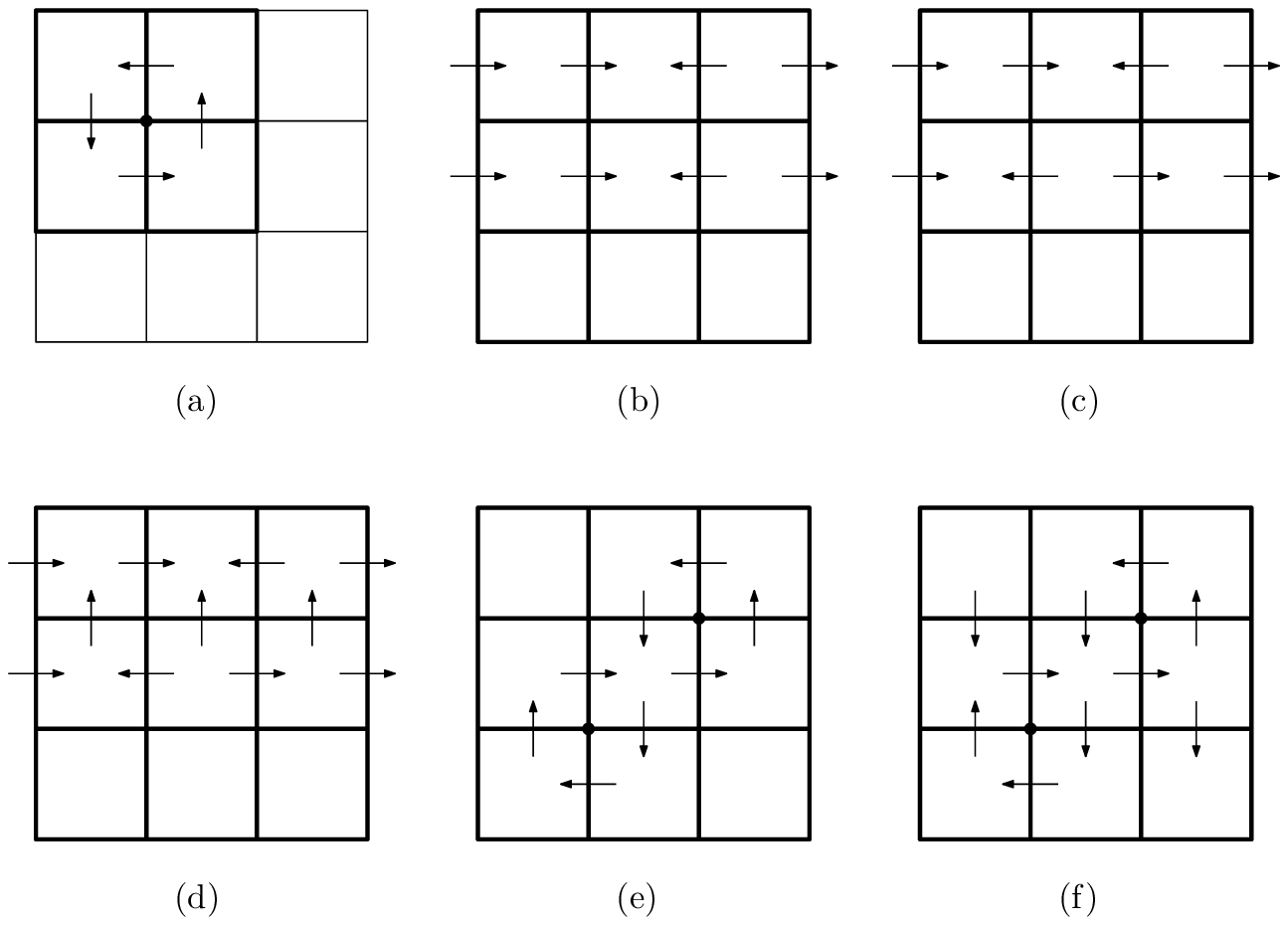}
\end{center}
\caption{Proof of Theorem \ref{thm:main}: 
\newline (a) a short loop  
\newline (b) two rows coinciding at all positions 
\newline (c) two rows coinciding at one and differing at two positions 
\newline (d) Lemma \ref{lem:rc_lem}(1): if two rows differ at two positions, then either there is a short loop or one row dominates the other
\newline (e) Lemma \ref{lem:num_sloops}: two short loops not between the same rows/columns 
\newline (f) Lemma \ref{lem:rc_lem}(2) applied to columns 1,2 and 2,3 in previous diagram}
\label{fig:main_proof}
\end{figure}

Let us introduce the notion of rows (or columns) \textit{coinciding or differing at a position}. Two rows, say $i$ and $i'$, coincide at a position, say between columns $j$ and $j'$, if $(i,j) \rightarrow (i,j')$ if and only if $(i',j) \rightarrow (i',j')$, and they differ at this position otherwise. E.g. in Fig.\ref{fig:main_proof}(b) rows 1 and 2 coincide at all positions, whereas in Fig.\ref{fig:main_proof}(c) they coincide at one and differ at two positions. 

We now introduce a few very helpful lemmas about the transition diagrams of zero-sum games satisfying our assumptions:

\begin{lemma} \label{lem:comb0}
Two columns (or rows) can have at most two short loops between them.
\end{lemma}

\begin{lemma} \label{lem:rc_lem}
\begin{enumerate}
\item If two rows (columns) differ at two positions, then there is a short loop between these rows (columns).
\item If two rows (columns) differ at all three positions, then there are precisely two short loops between them.
\item If two rows (columns) coincide at two positions, then there is a short loop between each of these rows (columns) and the third row (column). In particular the diagram has at least two short loop.
\item If two rows (columns) coincide at all three positions, then there are precisely two short loops between each of these rows (columns) and the third row (column). Then the diagram has precisely four short loops.
\end{enumerate}
\end{lemma}

\begin{proof}
For (1), note that every $2 \times 2$ block obtained by deleting one row and one column from a transition diagram has got either two arrows pointing in the same direction or contains and alternating cycle or a short loop. Assume that two rows (columns) differ at two positions (Fig.\ref{fig:main_proof}(c)). Since we don't allow alternating cycles, the only way a short loop between the two rows (columns) can be avoided is by having all arrows between them pointing in the same direction (Fig.\ref{fig:main_proof}(d)). But this case is ruled out by hypothesis (2) of the theorem. Hence there is a short loop between them. 

Essentially the same argument shows that statement (2) of the lemma holds. 

If two rows (columns) coincide at two positions, then since no column (row) is allowed to be dominated, each of these rows differs at two positions from the third row (column). Statement (3) then follows from statement (1). 

The same argument proves statement (4). The fact that the diagram then has precisely four short loops follows from Lemma \ref{lem:comb0} and the fact that there cannot be any short loops between the two rows (columns) that coincide at all positions.
\end{proof}

We can now proceed to grouping all possible transition diagrams by the number of short loops contained in them: 

\begin{lemma} \label{lem:num_sloops}
The transition diagram of a game as in Theorem \ref{thm:main} can only have between three and six short loops.
\end{lemma}

\begin{proof}
Note first that for any pair of rows (or columns) of a transition diagram at least one of the cases of Lemma \ref{lem:rc_lem} applies and we can make the following list of cases for two rows (columns), say $i$ and $j$:
\begin{itemize}
\item $i$ and $j$ coincide at 0 positions, then they have precisely 2 short loops between them.
\item $i$ and $j$ coincide at 1 position, then they have 1 or 2 short loops between them.
\item $i$ and $j$ coincide at 2 positions, then there is at most 1 short loop between them, and there is at least 1 short loop between each of them and the third row (column).
\item $i$ and $j$ coincide at 3 positions, then they have no short loops between them, and there are precisely 2 short loops between each of them and the third row (column).
\end{itemize}
It is clear from these that there cannot be a diagram without any short loops: pick any two rows and whichever of the above cases applies, it follows that the diagram has at least one short loop. 

Similarly, the transition diagram cannot have precisely one short loop. Assume for contradiction that (without loss of generality) there is a single short loop between rows 1 and 2. Now consider rows 2 and 3, which do not have a short loop between them. Then they must have at least 2 coinciding positions. But having 2 or more coinciding positions implies that there is also a short loop between rows 3 and 1, which contradicts the assumption. 

Now let us show that there is no transition diagram with precisely two short loops. Assume first, such diagram exists and both short loops are between the same two rows (or columns), say rows 1 and 2. Applying the above rules to rows 2 and 3 we see that either there have to be more short loops between rows 2 and 3 or between rows 3 and 1, both a contradiction. So the only left possibility is that the two short loops are neither between the same two rows nor columns. 

Here there are two cases to check: either both short loops run clockwise or one run clockwise and one runs anti-clockwise (any other configuration leads to a combinatorially equivalent diagram). Assume the short loops have different orientation. Without loss of generality we have the configuration shown in Fig.\ref{fig:main_proof}(e). By Lemma \ref{lem:rc_lem}(2) applied to columns 1,2 and 2,3 we get that $(1,1) \rightarrow(2,1)$ and $(2,3) \rightarrow(3,3)$ (Fig.\ref{fig:main_proof}(f)). But now Lemma \ref{lem:rc_lem}(1) applied to columns 1,3 implies that there is a third short loop. A similar chain of deductions shows that the case with both short loops having the same orientation also cannot happen. Hence a transition diagram with two short loops is not possible. 

At last, the upper bound of six short loops follows directly from Lemma \ref{lem:comb0}, which finishes the proof of the lemma.
\end{proof}

We can now state the final lemma of the proof of the main theorem. The proof of the lemma consists of easy (but somewhat tedious) deductions of the only possible combinatorial configurations for the transition diagrams and we do not provide complete details. Lemma \ref{lem:rc_lem} is very useful to reduce the number of diagrams that have to be checked.

\begin{lemma}
Up to combinatorial equivalence, there are precisely
\begin{itemize}
\item two non-equivalent transition diagrams with precisely three short loops,
\item fifteen non-equivalent transition diagrams with precisely four short loops,
\item five non-equivalent transition diagrams with precisely five short loops, 
\item one transition diagram with precisely six short loops,
\end{itemize}
that satisfy conditions (1)-(5) of the main theorem.
\end{lemma}

\begin{figure}
\begin{center}
\includegraphics{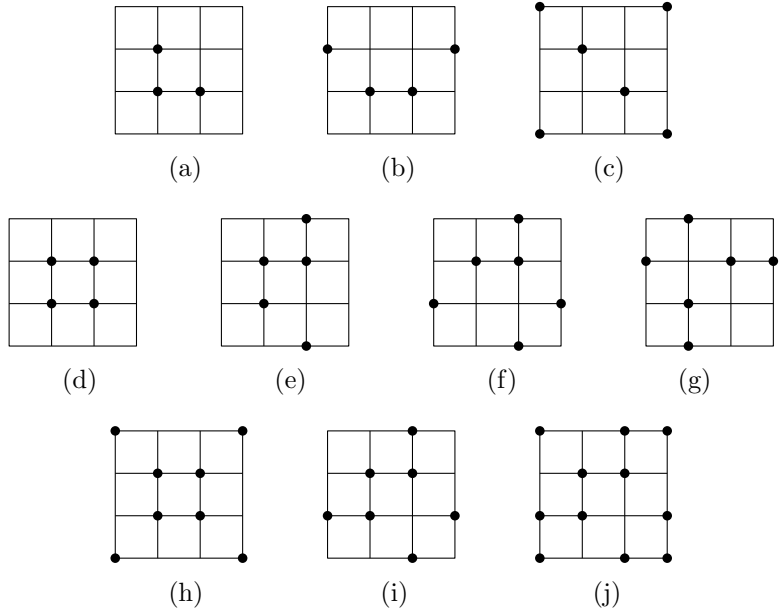}
\end{center}
\caption{The possible (non-equivalent) configurations for a transition diagram containing 3 (a-c), 4 (d-g), 5 (h-i) or 6 (j) short loops}
\label{fig:sloop_pos}
\end{figure}

\begin{proof}

Up to combinatorial equivalence, there are three ways in which three short loops can be positioned, see Fig.\ref{fig:sloop_pos}(a)-(c). It can be checked that only the first of these can give a transition diagram that satisfies (1)-(5), and there are two non-equivalent such diagrams.

Further, there are four ways to position four short loops (Fig.\ref{fig:sloop_pos}(d)-(g)), the first three of which admit five non-equivalent transition diagrams each, whereas the last one contradicts (1)-(5). 

The two ways to position five short loops (Fig.\ref{fig:sloop_pos}(h)-(i)) admit two and three non-equivalent transition diagrams satisfying (1)-(5). At last, applying Lemma \ref{lem:comb0} it is obvious that up to combinatorial equivalence the only way to position six short loops is the one shown in Fig.\ref{fig:sloop_pos}(j) and it is straightforward to check that there is only one possible transition diagram of this type.
\end{proof}

Together with Lemma \ref{lem:num_sloops}, this shows that there are precisely 23 transition diagram types satisfying (1)-(5). A list of the 23 diagram types and zero-sum game bimatrices realising them can be found in Appendix \ref{ap:types}, which finishes the proof of Theorem \ref{thm:main}. 
\end{proof}

\section{Quasi-Periodic Orbits}

In this last section of the analytic part of the paper we introduce the game-theoretic notion of quasi-periodicity and investigate the relation to its usual mathematical definition. This notion was first introduced in \cite{Rosenmuller1971}:


\begin{definition}
We say that an orbit of BR dynamics is \textbf{quasi-periodic} (in the game-theoretic sense), if its itinerary is periodic.
\end{definition}

Note that a priori this notion of quasi-periodicity of an orbit is different from the usual mathematical definition
(of an orbit which is dense in an invariant torus). However, we will show that the two notions are closely related. 

We consider the Hamiltonian dynamics on $S^3$ corresponding to a zero-sum game and its first return maps to the (two-dimensional) planes on which either $p(t)$ or $q(t)$ changes direction, i.e. where one of the $\BR$-correspondences is multivalued.

Let $S$ be such a plane, let $x \in S$ have a quasi-periodic orbit with (infinite) periodic itinerary $I = I(x)$, and let $\hat T$ be the first return map to $S$ (defined on the non-empty subset of $S$ of points whose orbits return to $S$). Note that $\hat T$ acts as a shift by a finite number of symbols on the itinerary of $x$. In particular there exists $n \geq 1$, such that $\hat{T}^n(x)$ has the same itinerary as $x$. Let us denote $T = \hat{T}^n$ so that each point in the $T$-orbit of $x$ has the same periodic itinerary $I$. 

\begin{lemma}
Let $U = \left\{ z \in S \colon I(z) = I(x) = I \right\}$ be the set of points in $S$ with itinerary $I$. Then $U$ is convex and $T(U) = U$. Moreover the restriction of $T$ to $U$, $T \colon U \rightarrow U$, is affine. 
\end{lemma}

\begin{proof}
Convexity follows from the convexity of the surfaces on which $p(t)$ or $q(t)$ change direction and the fact that the flow in each region $R_{ij}$ follows the 'rays' of a central projection. By the definition of $T$ we have that $T(U) \subseteq U$. We know (\cite{VanStrien2009a}) that $T$ is area-preserving and piecewise affine on $S$, so that $T(U) = U$. Since all points in $U$ have the same itinerary, it follows that $T \colon U \rightarrow U$ is indeed affine.
\end{proof}

\begin{theorem}
Let $x\in S$ correspond to a quasi-periodic orbit (in the game theoretical sense) of the Hamiltonian dynamics of a $3\times 3$ zero-sum bimatrix game, where $S$ is an indifference plane and $T$ is the return map to $S$, such that $I(T(x)) = I(x)$. Then one of the following holds:
\begin{enumerate}
\item The orbit of $x$ is periodic and $T^n(x) = x$ for some $n \geq 1$.
\item The $T$-orbit of $x$ lies on a $T$-invariant circle (more generally, depending on coordinates: an ellipse) and $x$ corresponds to a quasi-invariant orbit of the Hamiltonian dynamics (in the usual sense), i.e. its orbit under the flow is dense in an invariant torus. 
\end{enumerate}
In the second case $U = \left\{ z \in S \colon I(z) = I(x) \right\}$ is a disk and $T\colon U \rightarrow U$ is a rotation by an irrational angle (in suitable linear coordinates). Therefore in this case \textnormal{(2)} holds for \textit{every} $z\in U$.
\end{theorem}

\begin{proof}
Assume that $x$ is not periodic. We already know that $T \colon U \rightarrow U$ is a planar affine transformation of $U$. Since it is also an isometry and $T(U) = U$, the only possibility is that $T$ restricted to $U$ is a rotation by an irrational angle (any other kind of planar affine transformation satisfying these conditions would have $x$ as a periodic point). The result immediately follows. 
\end{proof}

The theorem shows that every quasi-periodic orbit (in the game-theoretic sense) is actually quasi-periodic in the usual sense. Conversely, every quasi-periodic orbit in the usual sense, which lies on a torus that only intersects the indifference surfaces along whole circles (and never just partially along an arc) is clearly quasi-periodic in the game theoretic sense. Throughout the rest of this paper, we will always refer to the game-theoretic definition, when using the notion of quasi-periodicity. 

From the argument above we can also immediately conclude: 

\begin{corollary}
If a Hamiltonian system induced by a $3 \times 3$ zero-sum bimatrix game has got an orbit with periodic itinerary, then it also has an actual periodic orbit. 
\end{corollary}

\part{Numerical Investigations and Conjectures}

In the second part of this paper we investigate the Hamiltonian dynamics induced by the BR dynamics numerically. We explore which types of orbit occur in these systems and how the combinatorial description given in the first part relates to these observations.

\section{Numerical Observations}

In this section we present some of our observations on the behaviour of BR dynamics for zero-sum games, mostly obtained from numerical experiments. The aspects we investigate are:
\begin{itemize}
\item The time fraction that different orbits spend in each of the regions $R_{ij}$.
\item The frequencies with which different orbits visit the regions $R_{ij}$ and the transition probabilities for transitions between regions. 
\item The different types of orbits that can occur and their itineraries (periodic, quasi-periodic, space-filling). 
\end{itemize}
The systems we consider are randomly generated examples of zero-sum games of different combinatorial types, for which we look at the induced Hamiltonian dynamics on level sets of $H$. For randomly chosen initial points we compute the orbits of the BR dynamics (more precisely, its induced Hamiltonian analogue) and study the time fractions spent in each region $R_{ij}$ and the frequencies, with which the orbits visit the regions. Especially with respect to the presented types of orbit we do not claim to give an exhaustive account of occuring types but rather a list of examples illustrating a few key concepts. 

Formally, for an orbit of the BR dynamics $(p(t),q(t)),~t > 0$ with itinerary $(i_0,j_0) \rightarrow (i_1,j_1) \rightarrow \ldots \rightarrow (i_k,j_k) \rightarrow \ldots$ and switching times $(t_n)$ we define
\begin{equation*}
P^{BR}_{ij}(n) = \frac{1}{t_n} \int_0^{t_n} \chi_{ij}(p(s),q(s)) ds~,
\end{equation*}
where $\chi_{ij}$ is the characteristic function of the region $R_{ij}$.

Alternatively we record the number of times, that each region is being visited by an orbit and compute the frequencies:
\begin{equation*}
Q_{ij}(n) = \frac{1}{n} \sum_{k=0}^{n-1} I_{ij}(i_k,j_k)~,~\text{ where } 
I_{ij}(i_k,j_k)=\begin{cases}
  1,  & \text{if } (i_k,j_k) = (i,j)\\
  0, & \text{otherwise.} \end{cases}
\end{equation*}
We write $P^{BR} = \left( P^{BR}_{ij} \right)_{i,j}$ and $Q = \left( Q_{ij} \right)_{i,j}$. 

Moreover throughout the following examples we look at orbits of the first return maps for the BR dynamics to certain surfaces of section. A convenient choice of such surface is a hyperplane on which either $p(t)$ or $q(t)$ changes direction, i.e. where one of the $\BR$-correspondences is multivalued. We will mostly use the surfaces where $q(t)$ changes direction: 
\begin{equation*}
S_{ij} = \left\{ (p,q) \in \Sigma : \{i,j\} \subset \BR_B(p) \right\}.
\end{equation*}

Let us now look at some examples:

\begin{example}[Uniquely ergodic case]\label{ex:erg}
{\rm
Let the zero-sum bimatrix game $(A,B)$ be given by
\begin{equation*}
A = \begin{pmatrix}
22 & 34 & -4 \\
7 & -32 & 16 \\
-53 & 96 & 23 
\end{pmatrix}~,~ B = -A~.
\end{equation*}

We numerically calculate orbits with itineraries of $10^4$ transitions for several hundreds of randomly chosen initial conditions. For all of these orbits, the evolution of $P^{BR}(n)$ and $Q(n)$ indicates a convergence to
\begin{equation*}
P^{BR} \approx 10^{-2} \times \begin{pmatrix}
13 & 5 & 27 \\
14 & 5 & 27 \\
3 & 1 & 5
\end{pmatrix}, \\
Q \approx 10^{-2} \times \begin{pmatrix}
12 & 9 & 19 \\
9 & 13 & 15 \\
10 & 5 & 8
\end{pmatrix}.
\end{equation*}

In Fig.\ref{fig:ex1_evolution}, the evolution of some of the $P^{BR}_{ij}(n)$ and $Q_{ij}(n)$ along an orbit is shown. 
\begin{figure}
\begin{center}
\includegraphics[width = 0.8\textwidth]{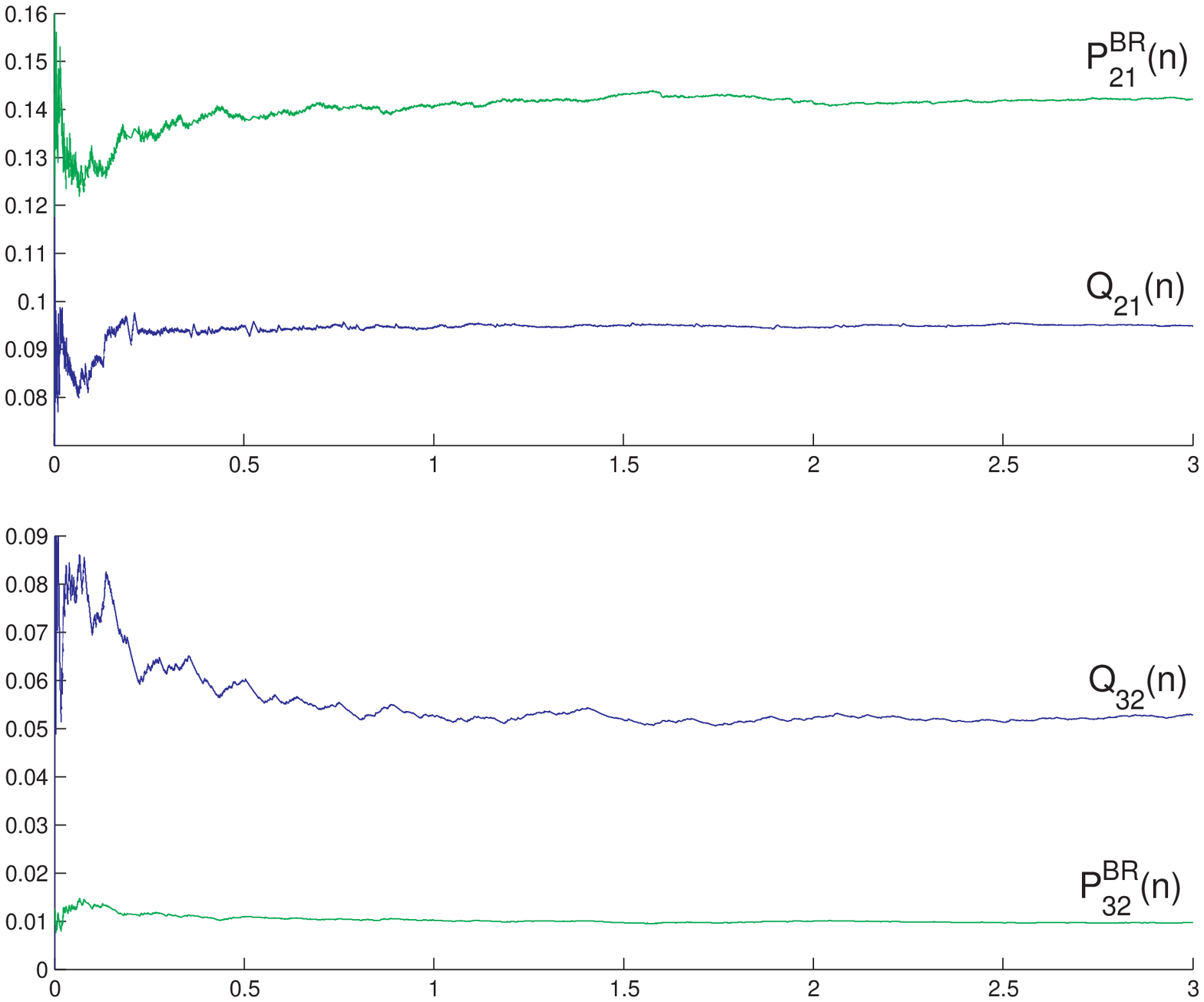}
\end{center}
\caption{Example \ref{ex:erg}: 
\newline The first plot shows the evolution of $P^{BR}_{21}(n)$ and $Q_{21}(n)$ along an orbit with itinerary of length $3\times 10^4$. A certain 'stabilisation' and convergence to the values $p_{21} \approx 0.14$ and $q_{21} \approx 0.09$ or small intervals containing these values can be observed. 
\newline The second plot shows the evolution of $P^{BR}_{32}(n)$ and $Q_{32}(n)$ along the same orbit. Here the observed limits (or limit intervals) are near the values $p_{32} \approx 0.01$ and $q_{32} \approx 0.05$. }
\label{fig:ex1_evolution}
\end{figure}
This or very similar statistical behaviour is observed for all sampled initial conditions. It seems to suggest that initial conditions with quasi-periodic orbits have zero or very small Lebesgue measure in the phase space of BR dynamics for this bimatrix game, as quasi-periodicity in all our experiments leads to very rapid convergence to certain frequencies. Most of the space seems to be filled with orbits that statistically resemble each other in the sense that they all visit certain portions of the space (the regions $R_{ij}$) with asymptotically equal (or very close) frequencies. The same seems to hold for the fraction of time spent in each region by the orbits.

\begin{figure}
\begin{center}
\includegraphics[width = \textwidth]{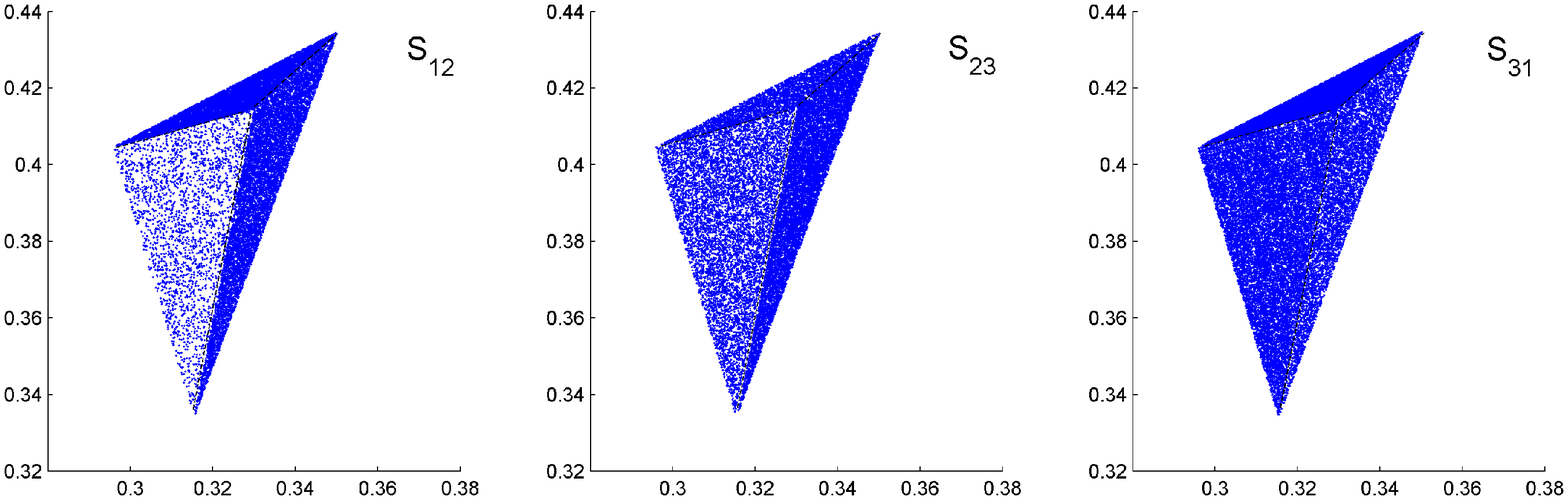}
\end{center}
\caption{Example \ref{ex:erg}: Typical orbit of induced Hamiltonian dynamics intersected with the three surfaces $S_{12}$, $S_{23}$ and $S_{31}$ (hyperplanes, where $q(t)$ changes direction). The original orbit has an itinerary of length $2 \times 10^6$. The three triangular regions in each of the $S_{ii'}$ correspond to the different possible transitions between regions $(i,j)$ and $(i',j)$ for $j=1,2,3$ and are indicated by dashed lines. (In other words the images show the orbit of the first return map to the surfaces of section $S_{ii'}$.) \newline
The different visiting frequencies of the regions can clearly be seen by the different densities of orbit points. The orbit points inside each triangular region seem to be uniformly distributed. }
\label{fig:ex1_section}
\end{figure}

Fig.\ref{fig:ex1_section} shows the intersections of an orbit of (\ref{eq:br_dynamics}) for this game with $S_{12}$, $S_{23}$ and $S_{31}$ (the hypersurfaces where $q(t)$ changes direction), i.e. the orbit of the first return map to these surfaces. Each $S_{ii'}$ consists of three triangular pieces, corresponding to the three pieces of hypersurface between regions $(i,j)$ and $(i',j)$ for $j=1,2,3$. Inside each of these triangles, the orbit seems to rather uniformly fill the space, suggesting ergodicity (of Lebesgue measure). If the BR dynamics had invariant tori, these would appear on all or some of these sections as invariant circles whose interior cannot be entered by orbits starting outside. Judging from the above observations, in this example they either don't exist or have very small radius. 
}\end{example}

\begin{example}[Space decomposed into ergodic and elliptic regions]\label{ex:qp}
{\rm
In this example we consider a bimatrix game, which is an element in a family of bimatrix games thoroughly studied in \cite{Sparrow2007} and \cite{VanStrien2009b}.  Let the zero-sum bimatrix game $(A,B)$ be given by
\begin{equation*}
A = \begin{pmatrix}
1 & 0 & \sigma \\
\sigma & 1 & 0 \\
0 & \sigma & 1 
\end{pmatrix}~,~ B = -A~, 
\end{equation*}
where $\sigma = \frac{\sqrt{5}-1}{2} \approx 0.618$ is the golden mean. 

Two types of orbit can be (numerically) observed for the BR dynamics of this game. The first type resembles the orbits in the previous example. The empirical frequencies $P^{BR}_{ij}(n)$ and $Q_{ij}(n)$ along such orbits initially behave erratically but seem to suggest convergence to certain values or narrow ranges of values, which are the same for all such orbits:

\begin{equation*}
P^{BR} \approx 10^{-2} \times \begin{pmatrix}
9 & 11 & 13 \\
13 & 9 & 11 \\
11 & 13 & 9
\end{pmatrix}, \\
Q \approx 10^{-2} \times \begin{pmatrix}
11 & 13 & 9 \\
9 & 11 & 13 \\
13 & 9 & 11
\end{pmatrix}.
\end{equation*}

It can be observed that the values of $Q(n)$ are perhaps less erratic and in most of our experiments they seem to converge faster than those of $P^{BR}(n)$. As an example, the evolution of $P^{BR}_{32}(n)$ and $Q_{32}(n)$ along a typical orbit can be seen in Fig.\ref{fig:ex2_p32_erg}. 

\begin{figure}
\begin{center}
\includegraphics[width = 0.8\textwidth]{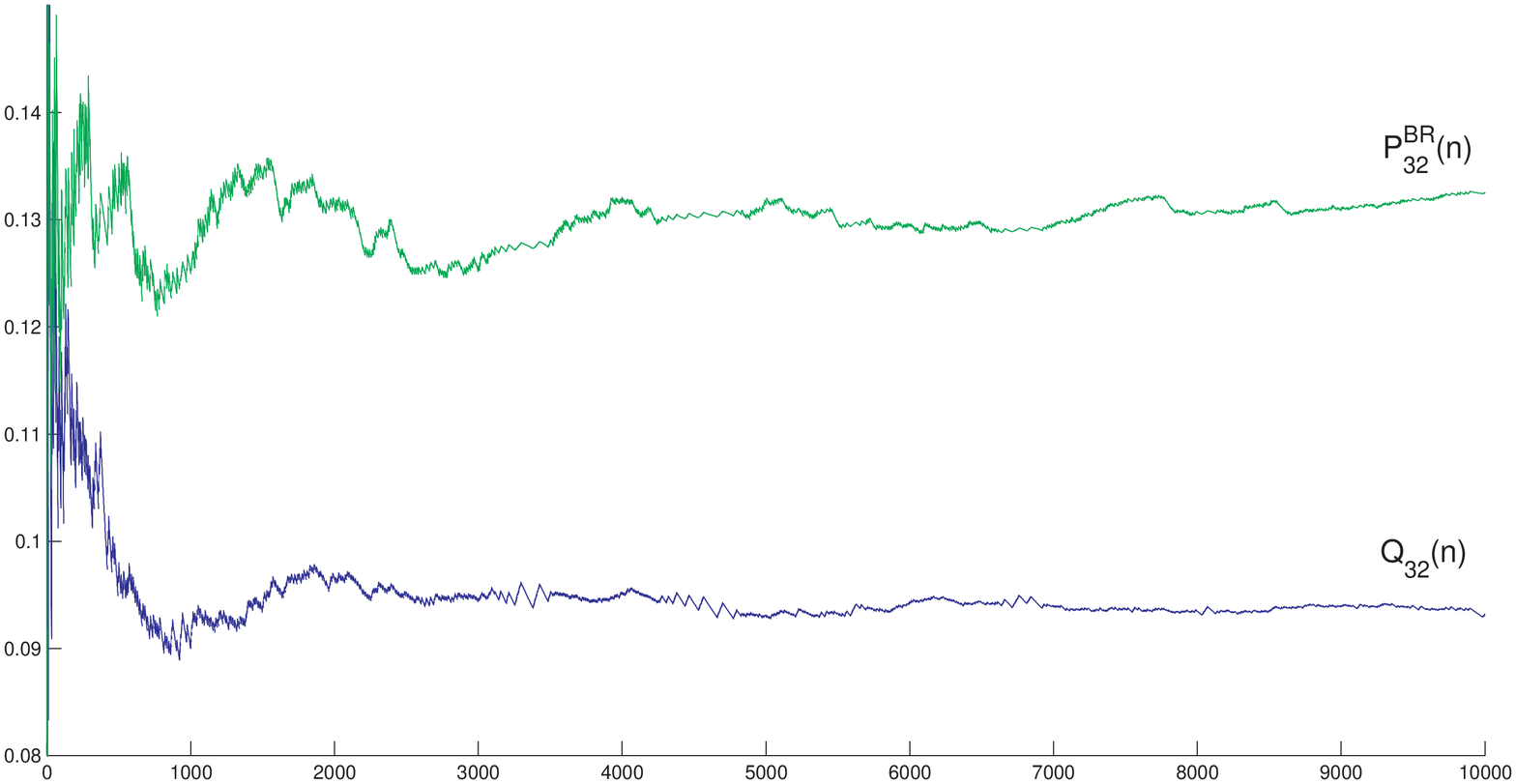}
\end{center}
\caption{Example \ref{ex:qp}: The evolution of $P^{BR}_{32}(n)$ and $Q_{32}(n)$ along a typical orbit outside of the invariant torus. }
\label{fig:ex2_p32_erg}
\end{figure}

As in the previous example, in Fig.\ref{fig:ex2_section} we show the intersection of one such orbit with the surfaces $S_{ii'}$. Once again the orbit points have a certain seemingly uniform density inside each region, but here they leave out an elliptical region on each of the hypersurfaces. This invariant region consists of invariant circles, formed by quasi-periodic orbits of the system (the second type of observed orbits). The center of the circles corresponds to an actual periodic orbit, i.e. an elliptic fixed point of the return map to one of these surfaces. See \cite{VanStrien2009b} for an explicit analytic investigation of this (which is made possible by the high symmetry of this particular bimatrix game). 

\begin{figure}
\begin{center}
\includegraphics[width = \textwidth]{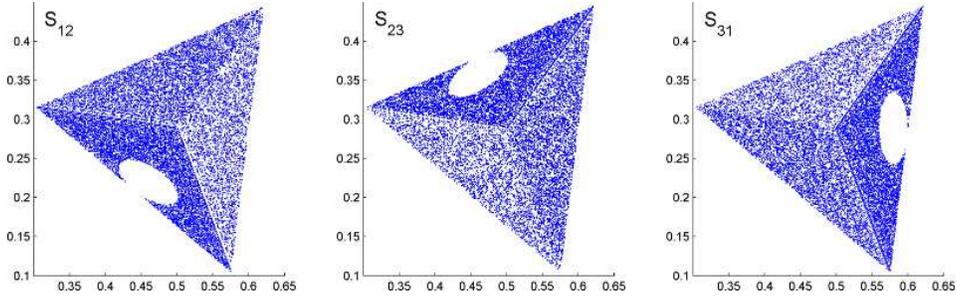}
\end{center}
\caption{Example \ref{ex:qp}: Typical orbit of induced Hamiltonian dynamics intersected with the three hyperplanes, where $q(t)$ changes direction. The original orbit has an itinerary of length $10^6$. The different visiting frequencies of the regions can clearly be seen by the different densities of orbit points. The orbit points inside each triangular region seem to be uniformly distributed but leave out elliptical regions in some of the regions. These are filled with quasi-periodic orbits forming invariant circles.}
\label{fig:ex2_section}
\end{figure}

The invariant circles in the elliptical region correspond to invariant tori in the BR dynamics. Their itinerary is periodic with period 6:
\begin{equation*}
(1,1) \rightarrow (1,2) \rightarrow (2,2) \rightarrow (2,3) \rightarrow (3,3) \rightarrow (3,1) \rightarrow (1,1) \rightarrow \ldots. 
\end{equation*}

\begin{figure}
\begin{center}
\includegraphics[width = \textwidth]{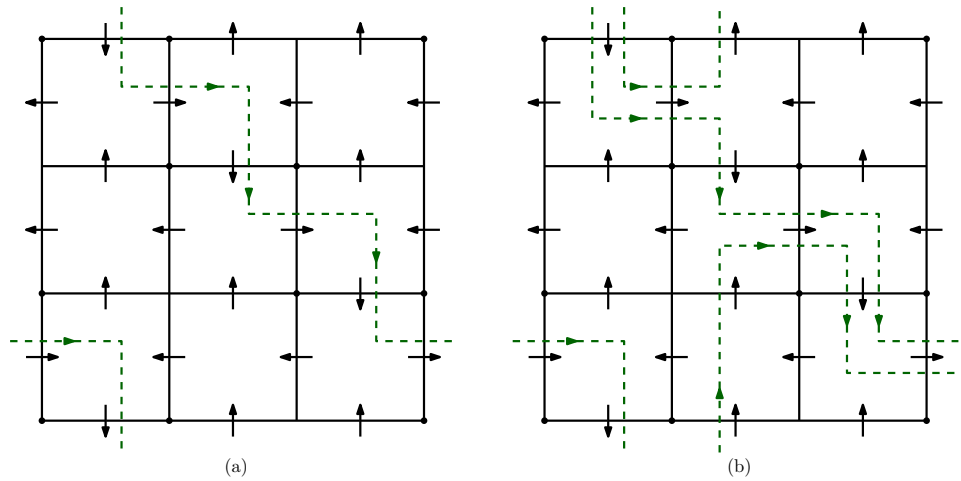}
\end{center}
\caption{(a) Example \ref{ex:qp}: The transition diagram of the bimatrix game. The periodic itinerary of the quasi-periodic orbits forming an invariant torus is indicated by a dashed line.
\newline (b) Example \ref{ex:qp_loop}: The transition diagram is the same as in the previous example. However, there is an invariant torus of quasi-periodic orbits with a more complicated periodic itinerary.}
\label{fig:type6_orbits}
\end{figure}

Fig.\ref{fig:type6_orbits}(a) shows the transition diagram for this bimatrix game. The periodic itinerary is indicated by a dashed line as a loop in the transition diagram. The empirical frequencies along such quasi-periodic orbits converge to
\begin{equation*}
P^{BR} = Q =  \begin{pmatrix}
\frac{1}{6} & \frac{1}{6} & 0 \\
0 & \frac{1}{6} & \frac{1}{6} \\
\frac{1}{6} & 0 & \frac{1}{6}
\end{pmatrix}.
\end{equation*}

Fig.\ref{fig:ex2_section} suggests that there are no other invariant tori for this system, i.e. no open set of initial conditions outside of the visible elliptical regions, whose orbits are all quasi-periodic.
}\end{example}

A question that arises naturally from the above example is the following: does an invariant torus of quasi-periodic orbits always have a 'simple' periodic itinerary as the above? Are the periods of such elliptic islands necessarily equal to 6? As the next example shows, the situation can indeed be more complicated and less simple paths through the transition diagram are possible candidates for the periodic itinerary of quasi-periodic orbits in an invariant torus. 

\begin{example}[Quasi-periodic behaviour with itineraries of higher period]\label{ex:qp_loop}
{\rm
Consider the bimatrix game $(A,B)$ with
\begin{equation*}
A = \begin{pmatrix}
84 & -37 & 10 \\
24 & 33 & -14 \\
-26 & 9 & 20
\end{pmatrix}~,~ B = -A~.
\end{equation*}

\begin{figure}
\begin{center}
\includegraphics[width = \textwidth]{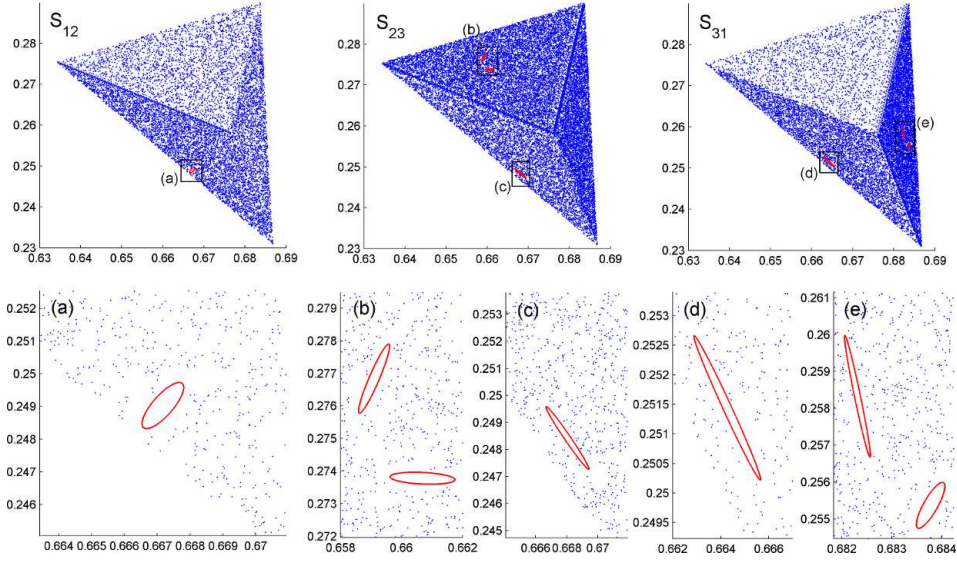}
\end{center}
\caption{Example \ref{ex:qp_loop}: Two orbits (their intersections with the surfaces $S_{ii'}$) are shown: one which stochastically fills most of the space and one which lies in an invariant torus. The latter intersects the first surface once and the other two three times each. The intersections of the (very thin) torus with the surfaces are marked with rectangles and shown magnified in the bottom row.}
\label{fig:erg_qp_orbits}
\end{figure}

Generally, the observations here coincide with Example \ref{ex:erg}. However, one can detect a (quite thin) invariant torus. Fig.\ref{fig:erg_qp_orbits} shows a typical orbit stochastically filling most of the space. In the bottom row of the same figure, the regions marked by rectangles are enlarged to reveal a thin invariant torus. These quasi-periodic orbits intersect $S_{12}$ once, $S_{23}$ and $S_{31}$ three times each. The orbits look essentially like the quasi-periodic orbits in the previous example, but with an extra loop added. The itineraries are periodic with period 13, where each period is of the form
$$
\begin{array}{llll}
(1,1)\rightarrow (1,2) &                                   &\rightarrow (2,2)\rightarrow (2,3)\rightarrow (3,3)\rightarrow (3,1)\rightarrow \\
(1,1)\rightarrow (1,2) & \!\!\!\!  \rightarrow (3,2)   \!\!\!     & \rightarrow (2,2)\rightarrow (2,3)\rightarrow (3,3)\rightarrow (3,1)\rightarrow (1,1). 
\end{array}
$$
In Fig.\ref{fig:type6_orbits}(b), this itinerary is shown as a loop in the transition diagram. The example demonstrates that combinatorially more complicated quasi-periodic orbits are possible for open sets of initial conditions in the BR dynamics of zero-sum games. 
}\end{example}

The next example shows an even more complex quasi-periodic structure and gives numerical evidence for more subtle and involved effects than those observed above:

\begin{example}[Coexistence of different elliptic behaviour]\label{ex:arnold_diff}
{\rm
Let us now consider the bimatrix game $(A,B)$ with 
\begin{equation*}
A = \begin{pmatrix}
-92 & 18 & 52 \\
62 & -37 & -33 \\
-10 & 9 & -18\end{pmatrix}~,~ B = -A~.
\end{equation*}

\begin{figure}
\begin{center}
\includegraphics[width = \textwidth]{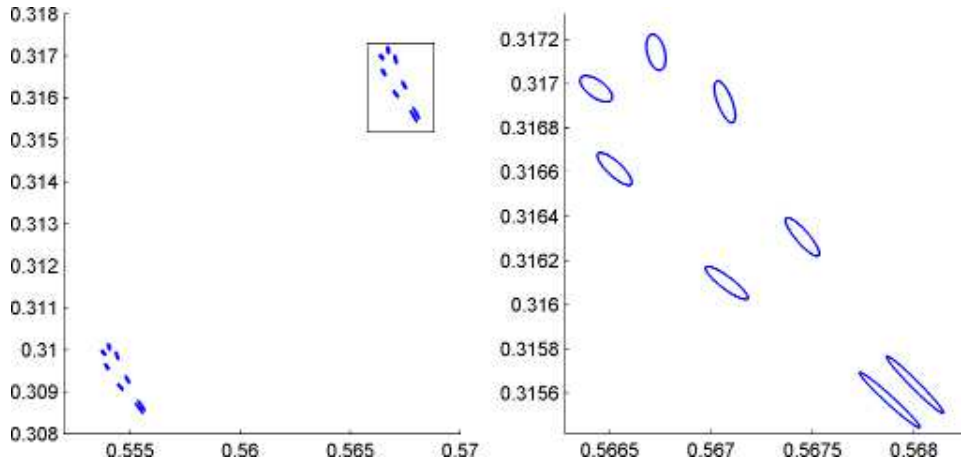}
\end{center}
\caption{Example \ref{ex:arnold_diff}: The invariant torus (corresponding to quasi-periodic orbits) intersects two of the hypersurfaces where $q(t)$ changes direction as shown on the left, while not intersecting the third hypersurface at all. The right image is a magnification of the region marked by a rectangle.}
\label{fig:type12_qp_orbit}
\end{figure}

As in all the previous examples, the largest part of the phase space of the BR dynamics seems to be filled with orbits which stochastically fill most of the space and along which the frequency distributions $P^{BR}(n)$ and $Q(n)$ seem to converge to certain (orbit-independent) values. Again, an invariant torus can be found. It is more complicated than those observed in the other examples (see Fig.\ref{fig:type12_qp_orbit}). The orbits forming this invariant torus are quasi-periodic and have an itinerary of period 60. Its structure suggests a generalisation of the type of itinerary observed in Example \ref{ex:qp_loop}. It consists of a sequence of blocks of the following two forms:
\begin{align*}
a &= \left( (1,1) \rightarrow (3,1) \rightarrow (2,1) \rightarrow (2,2) \rightarrow (3,2) \rightarrow (3,3) \rightarrow (1,3) \rightarrow (1,1) \right) \text { and } \\
b &= \left( (1,1) \rightarrow (3,1) \rightarrow (2,1) \rightarrow (2,2) \rightarrow (3,2) \rightarrow (3,3) \rightarrow (1,3) \rightarrow (1,2) \rightarrow (1,1) \right).
\end{align*}

\begin{figure}
\begin{center}
\includegraphics[width = \textwidth]{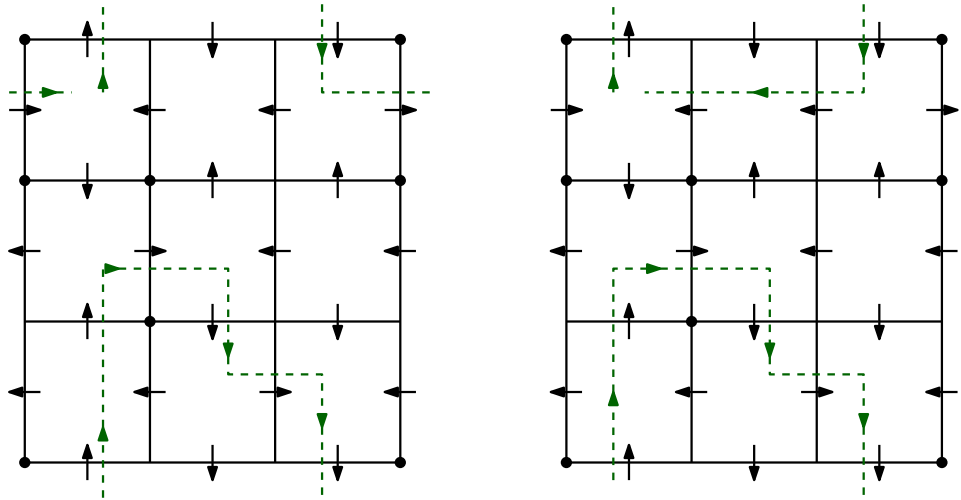}
\end{center}
\caption{Example \ref{ex:arnold_diff}: The two block types $a$ and $b$ of the itinerary for the quasi-periodic orbits in this system. Note that they differ by one step only.}
\label{fig:type12_blocktypes}
\end{figure}

The two blocks are shown as pathes in the transition diagram in Fig.\ref{fig:type12_blocktypes}. Each period of the itinerary of orbits in the invariant torus then has the form 
\begin{equation*}
a \rightarrow a \rightarrow b \rightarrow a \rightarrow b \rightarrow b \rightarrow a \rightarrow b.
\end{equation*}
As in the previous example the two blocks are the same except for one element (in the previous example the itinerary consisted of two blocks concatenated alternatingly).

\begin{figure}
\begin{center}
\includegraphics[width = \textwidth]{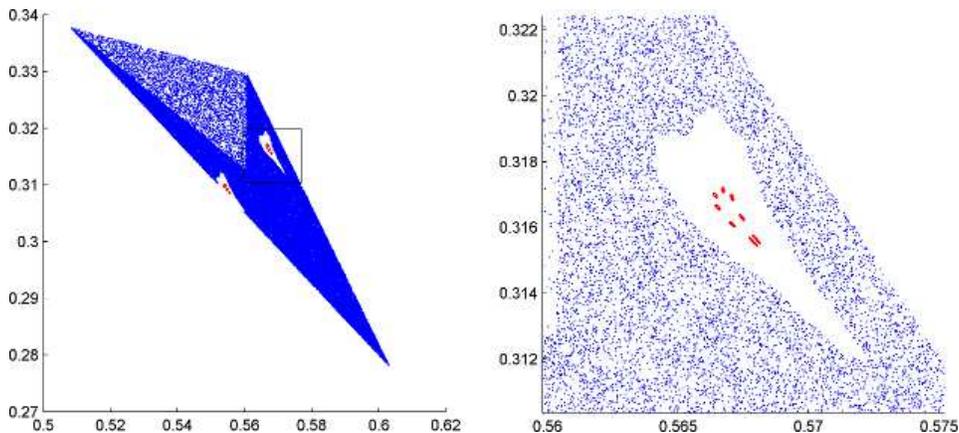}
\end{center}
\caption{Example \ref{ex:arnold_diff}: The left image shows a typical orbit that seems to fill almost all of the space. The right image shows a magnified view of the 'heart-shaped' region spared out by this orbit. Also both images show the invariant circles of a quasi-periodic orbit inside the 'heart-shaped' region.}
\label{fig:type12_erg_orbit}
\end{figure}

In Fig.\ref{fig:type12_erg_orbit} the intersection of an orbit outside of the invariant torus with one of the hypersurfaces is shown together with a quasi-periodic orbit. The orbit seems to have essentially the same property of filling the space outside the invariant torus, as in the previous examples. However, a closer look at a neighbourhood of the invariant circles (see the right part of Fig.\ref{fig:type12_erg_orbit}) reveals that the orbit not only misses out the invariant circles but also a certain 'heart-shaped' region surrounding these. The investigation of orbits with initial conditions in this set shows a range of effects not observed in any of the previous examples.

\begin{figure}
\begin{center}
\includegraphics[width = \textwidth]{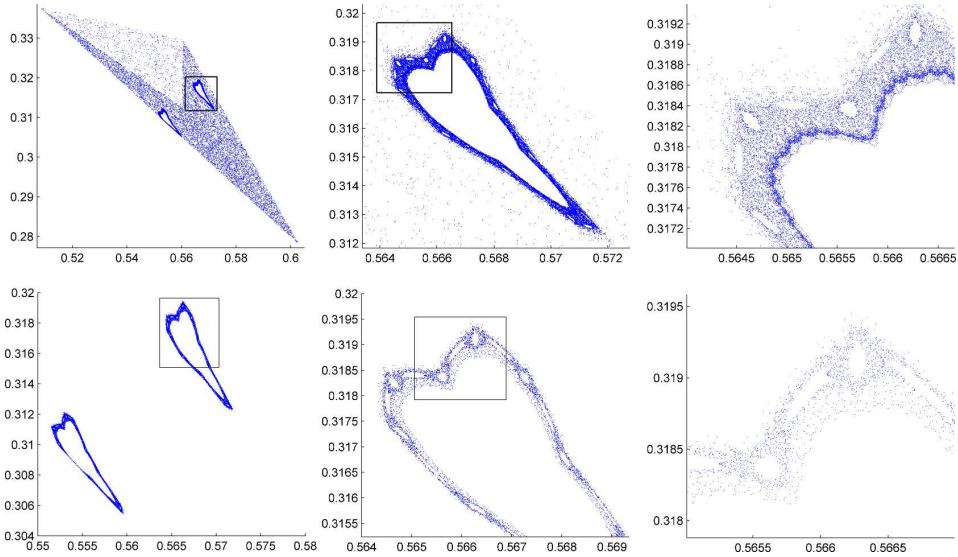}
\end{center}
\caption{Example \ref{ex:arnold_diff}: Two different orbits with initial conditions in the 'heart-shaped' region. Regions of stochastic motion as well as invariant islands of periodic and quasi-periodic orbits are clearly visible. The first orbit spends a long time in the 'heart-shaped' region before it diffuses into the larger 'ergodic' part of the space.}
\label{fig:type12_orbit_examples}
\end{figure}

\begin{figure}
\begin{center}
\includegraphics[width = \textwidth]{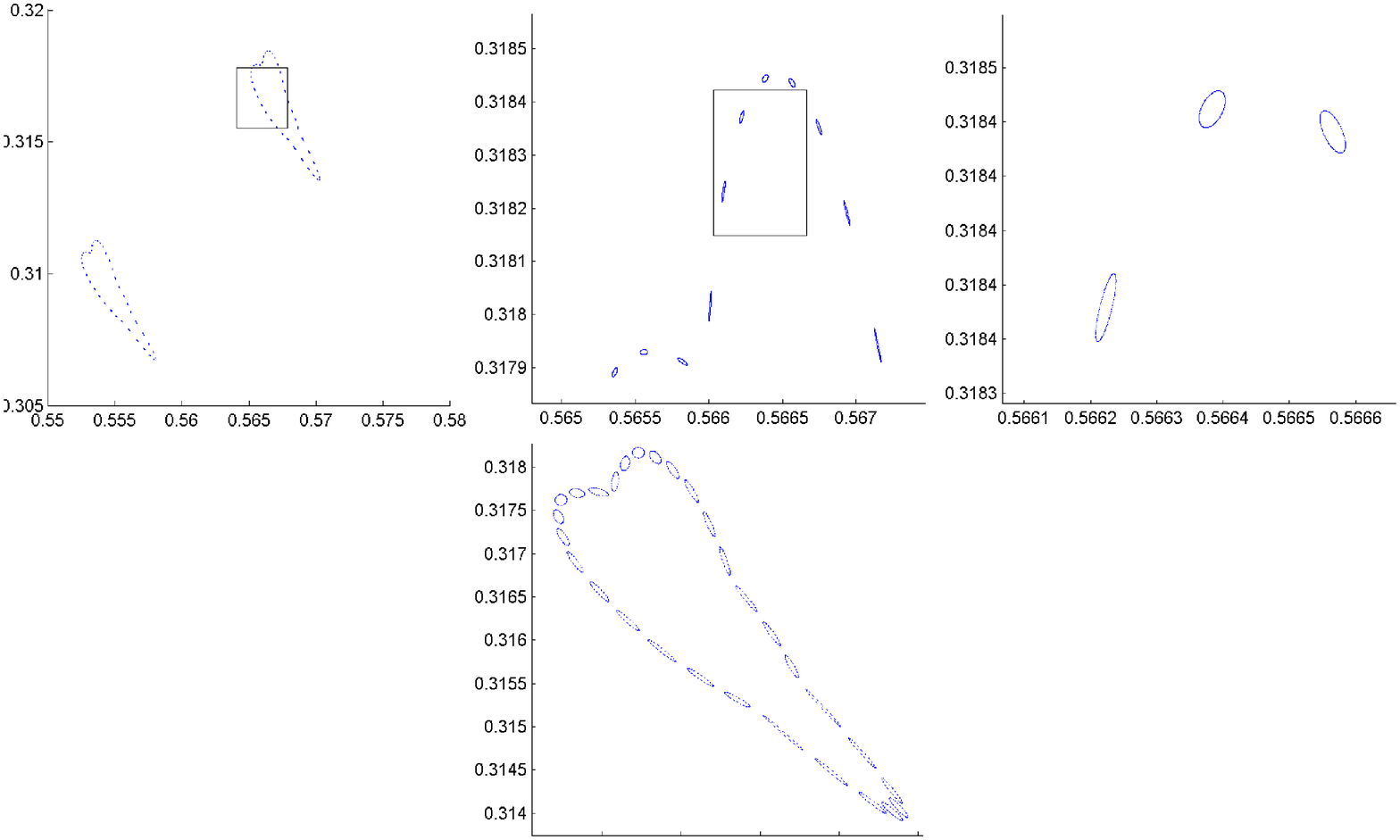}
\end{center}
\caption{Example \ref{ex:arnold_diff}: Two different quasi-periodic orbits (invariant tori for the BR dynamics) of different periods.}
\label{fig:type12_per_orbits}
\end{figure}

\begin{figure}
\begin{center}
\includegraphics[width = 0.8\textwidth]{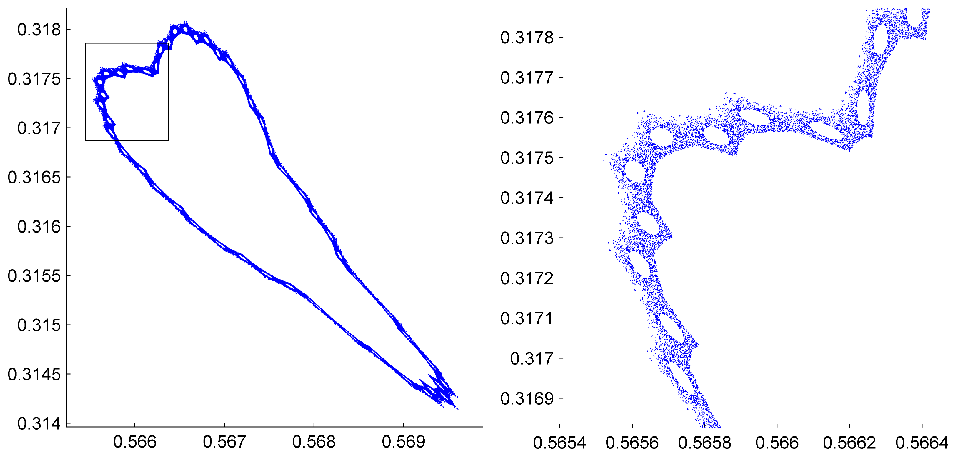}
\end{center}
\caption{Example \ref{ex:arnold_diff}: An orbit restricted to an invariant heart-shaped annulus.}
\label{fig:type12_restricted_orbit}
\end{figure}

Several different orbits with initial conditions in this region can be seen in Fig.\ref{fig:type12_orbit_examples}. The orbit points show complicated structures, revealing a large number of 'stochastic' regions as well as invariant regions of periodic orbits of high periods and corresponding quasi-periodic orbits (Fig.\ref{fig:type12_per_orbits} shows some examples of such quasi-periodic orbits of different higher periods). Some of these orbits spend very long times (itineraries of length $10^6$ and more) in the heart-shaped region before diffusing into the much larger 'stochastic' rest of the space. On the other hand we observe orbits that stochastically fill (heart-shaped) annuli leaving out islands of quasi-periodic orbits. These annuli seem to be invariant for the dynamics (see Fig.\ref{fig:type12_restricted_orbit}). 

Altogether the observations described above strongly indicate the occurrence of {\lq}Arnol'd diffusion{\rq}:
the coexistence of  a family of invariant annuli, which contain regions of stochastic (space-filling) motion and islands of further periodic orbits and invariant circles (quasi-periodic orbits). 
}\end{example}

\section{Conclusion and Discussion}

We would like to propose some open questions for further investigation.

\begin{enumerate}
\item Does the Hamiltonian system induced by a $3\times 3$ zero-sum bimatrix game always have quasi-periodic orbits / invariant tori? Example \ref{ex:erg} suggests that it is possible to have topological mixing and that the Lebesgue measure is ergodic. However, this might be due to the limited resolution of our numerical simulations and images. 

\item Are there orbits which are dense outside of the elliptic regions? Are almost all orbits outside of the elliptic regions dense? 

\item Example \ref{ex:arnold_diff} suggests that the system has infinitely many elliptic islands corresponding to quasi-periodic orbits of different periods. The pictures of orbits (e.g. Fig.\ref{fig:type12_orbit_examples} and Fig. \ref{fig:type12_restricted_orbit}) show many regions that could potentially contain such elliptic islands of quasi-periodic orbits of different periods. All regions that we investigated for this property actually revealed quasi-periodic orbits.

\item Given a specific bimatrix example, are there a finite number of blocks, so that the itinerary of any orbit on an elliptic island is periodic with each period being a (finite) concatenation of these blocks? The examples we looked at suggest the answer to be positive. 
\end{enumerate}

In Part 1 of this paper we assigned to a Hamiltonian system a transition diagram, giving a necessary condition on the itinerary of orbits. If we were able to develop some kind of 'admissibility condition' (i.e. a sufficiency condition), we could perhaps obtain results such as 'density of periodic orbits', in the same way as was done for quadratic maps of the interval. 

In Part 2 we demonstrated that this class of dynamics is sufficiently rich to mimic many of the intricacies of smooth Hamiltonian systems. In spite of the many open questions that still remain, it appears that these piecewise affine Hamiltonian systems could provide a new way of gaining insight into global dynamics of Hamiltonian systems. 

\bibliographystyle{plain}
\bibliography{GameTheory}

\begin{thebibliography}{10}

\bibitem{Aubin1984}
Jean-Pierre Aubin and Arrigo Cellina.
\newblock {\em {Differential Inclusions}}.
\newblock Springer, Berlin, 1984.

\bibitem{Berger2005}
Ulrich Berger.
\newblock {Fictitious play in $2 \times n$ games}.
\newblock {\em Journal of Economic Theory}, 120:139--154, 2005.

\bibitem{Brown1951}
George~W Brown.
\newblock {Iterative solution of games by fictitious play}.
\newblock {\em Activity analysis of production and allocation}, 13:374--376,
  1951.

\bibitem{MR2368310}
M.~di~Bernardo, C.~J. Budd, A.~R. Champneys, and P.~Kowalczyk.
\newblock {\em Piecewise-smooth dynamical systems}, volume 163 of {\em Applied
  Mathematical Sciences}.
\newblock Springer-Verlag London Ltd., London, 2008.

\bibitem{Hofbauer1995}
Josef Hofbauer.
\newblock {Stability for the Best Response Dynamics}.
\newblock August 1995.

\bibitem{MR1789550}
Markus Kunze.
\newblock {\em Non-smooth dynamical systems}, volume 1744 of {\em Lecture Notes
  in Mathematics}.
\newblock Springer-Verlag, Berlin, 2000.

\bibitem{MR2103797}
Remco~I. Leine and Henk Nijmeijer.
\newblock {\em Dynamics and bifurcations of non-smooth mechanical systems},
  volume~18 of {\em Lecture Notes in Applied and Computational Mechanics}.
\newblock Springer-Verlag, Berlin, 2004.

\bibitem{NashJohnForbes1951}
John~Forbes Nash.
\newblock {Non-Cooperative Games}.
\newblock {\em Annals of Mathematics}, 54(2):286--295, 1951.

\bibitem{Robinson1951}
Julia Robinson.
\newblock {An iterative method of solving a game}.
\newblock {\em Annals of Mathematics}, 54(2):296--301, 1951.

\bibitem{Rosenmuller1971}
Joachim Rosenm\"{u}ller.
\newblock {\"{U}ber Periodizit\"{a}tseigenschaften Spieltheoretischer
  Lernprozesse}.
\newblock {\em Zeitschrift f\"{u}r Wahrscheinlichkeitstheorie und Verwandte
  Gebiete}, 17(4):259--308, December 1971.

\bibitem{Sparrow2007}
Colin Sparrow, Sebastian van Strien, and Christopher Harris.
\newblock {Fictitious play in $3\times 3$ games: The transition between
  periodic and chaotic behaviour}.
\newblock {\em Games and Economic Behavior}, November 2007.

\bibitem{VanStrien2009a}
Sebastian van Strien.
\newblock {A new class of Hamiltonian flows with random-walk behavior
  originating from zero-sum games and Fictitious Play}.
\newblock {\em Preprint}, 2009.

\bibitem{VanStrien2009b}
Sebastian van Strien and Colin Sparrow.
\newblock {Fictitious Play in $3\times 3$ Games: chaos and dithering
  behaviour}.
\newblock {\em Preprint (to appear in Games and Economic Behavior)}, 2009.

\end{thebibliography}

\newpage
\appendix
\section{The 23 transition diagram types} \label{ap:types}
We list all 23 transition diagram types as in Theorem \ref{thm:main} sorted by the number of short loops contained in them, together with the respective matrices $A$, such that the bimatrix games $(A,-A)$ realise the given diagram types. 

\begin{figure}[b]
\begin{center}
\includegraphics[width = 0.88\textwidth]{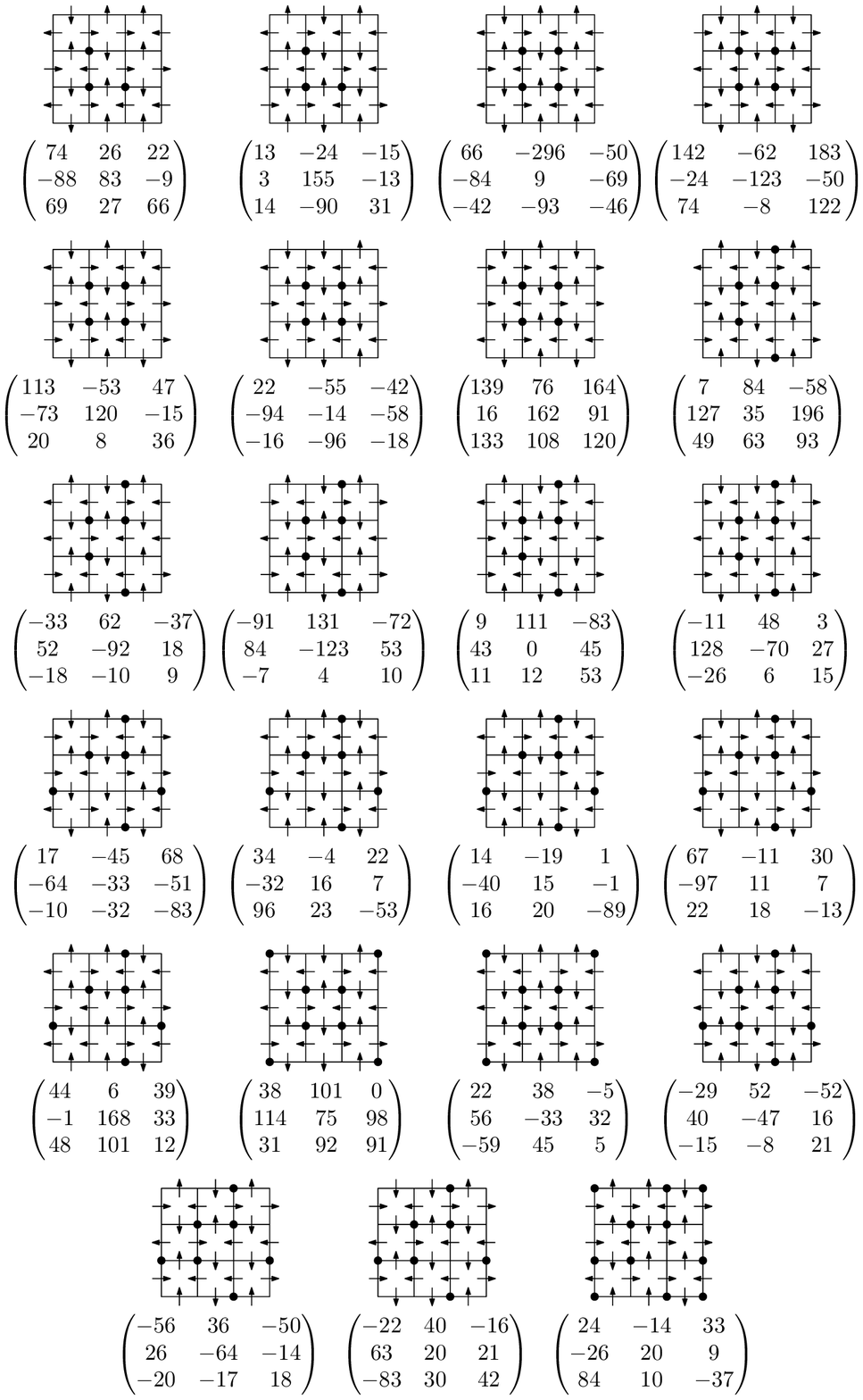}
\end{center}
\end{figure}

\end{document}